\documentclass[english]{article}
\usepackage[T1]{fontenc}
\usepackage[utf8]{inputenc}
\usepackage{lmodern}
\usepackage{amssymb,authblk}
\usepackage{mathtools}
\usepackage{amsthm}
\usepackage{stmaryrd}
\usepackage{mathrsfs}
\usepackage{amsfonts}
\usepackage{tdsfrmath}
\usepackage{faktor}
\usepackage{xcolor}
\usepackage{tcolorbox}
\usepackage[all,cmtip, color,matrix,arrow]{xy} 
\usepackage[linewidth=1pt,
middlelinecolor= black,
middlelinewidth=0.4pt,
roundcorner=1pt,
topline = false,
rightline = false,
bottomline = false,
rightmargin=0pt,
skipabove=0pt,
skipbelow=0pt,
leftmargin=-1cm,
innerleftmargin=1cm,
innerrightmargin=0pt,
innertopmargin=0pt,
innerbottommargin=0pt]{mdframed}

\usepackage{array}
\usepackage{nicematrix}
\usepackage[a4paper]{geometry}
\usepackage{shuffle}
\usepackage[english]{babel}
\usepackage{subcaption}
\usepackage[plainpages=false,pdfpagelabels,pagebackref]{hyperref} 
\usepackage{tikz-cd}
\usepackage{fp}
\usepackage{ifthen}
\usepackage{calc}
\usetikzlibrary{automata, positioning, arrows}

\theoremstyle{definition}
\newtheorem{defi}{Definition}[section]
\newtheorem{Eg}{\textbf{Example}}[section]

\newtheorem{Rq}{\textbf{Remark}}[section]
\newtheorem{Not}{\textbf{Notation}}[section]
\newtheorem*{Eg*}{Example}
\newtheorem*{defi*}{Definition}
\newtheorem*{Rq*}{\textbf{Remark}}

\newtheorem{innercustomgeneric}{\customgenericname}
\providecommand{\customgenericname}{}
\newcommand{\newcustomtheorem}[2]{%
	\newenvironment{#1}[1]
	{%
		\renewcommand\customgenericname{#2}%
		\renewcommand\theinnercustomgeneric{##1}%
		\innercustomgeneric
	}
	{\endinnercustomgeneric}
}
\newcustomtheorem{customdefi}{Definition}
\newcustomtheorem{customrq}{Remark}
\newcustomtheorem{customEgs}{Examples}
\newcustomtheorem{customEg}{Example}

\theoremstyle{plain}
\newtheorem{Prop}[defi]{Proposition}
\newtheorem{Lemme}[defi]{Lemma}
\newtheorem{Cor}[defi]{Corollary} 
\newtheorem{thm}[defi]{Theorem}
\newtheorem*{thm*}{Theorem}
\newtheorem*{Lemme*}{Lemma}
\newtheorem*{Prop*}{Proposition}

\newtheorem{innercustomgenerictwo}{\customgenericname}
\providecommand{\customgenericname}{}
\newcommand{\newcustomtheoremplain}[2]{%
	\newenvironment{#1}[1]
	{%
		\renewcommand\customgenericname{#2}%
		\renewcommand\theinnercustomgenerictwo{##1}%
		\innercustomgenerictwo
	}
	{\endinnercustomgeneric}
}
\newcustomtheoremplain{customthm}{Theorem}
\newcustomtheoremplain{customlemma}{Lemma}
\newcustomtheoremplain{customprop}{Proposition}

\newcommand{\IEM}[2]{\llbracket #1,#2 \rrbracket}

\newcommand{\K}{\ensuremath{\mathbb{K}}}
\newcommand{\A}{\ensuremath{\mathcal{A}}}

\newcommand{\qsh}{quasi-shuffle}
\newcommand{\sh}{shuffle}

\NewDocumentCommand{\fourletters}{ m  m m m }{\ensuremath{\mathbf{#1}~\mathbf{#2}~\mathbf{#3}~\mathbf{#4}}}
\NewDocumentCommand{\threeletters}{ m m m }{\ensuremath{\mathbf{#1}~\mathbf{#2}~\mathbf{#3}}}
\NewDocumentCommand{\twoletters}{ m m }{\ensuremath{\mathbf{#1}~\mathbf{#2}}}

\NewDocumentCommand{\MZV}{ }{Multiple Zeta Value}
\NewDocumentCommand{\MZVs}{ }{Multiple Zeta Values}
\NewDocumentCommand{\AZV}{}{Arborified Zeta Value}
\NewDocumentCommand{\AZVs}{}{Arborified Zeta Values}
\NewDocumentCommand{\Shin}{}{Shintani}
\NewDocumentCommand{\SZV}{}{\Shin{} Zeta Value}
\NewDocumentCommand{\SZVs}{}{\Shin{} Zeta Values}
\NewDocumentCommand{\Sch}{}{Schroeder}

\DeclareMathOperator{\ew}{\emptyset} 
\DeclareMathOperator{\QSh}{Qsh}
\DeclareMathOperator{\Sh}{Sh}
\DeclareMathOperator{\Suum}{\mathcal{S}}
\DeclareMathOperator{\Suumcv}{\Suum_{\N^*}^{\rm conv}}
\DeclareMathOperator{\dt}{dt}
\DeclareMathOperator{\du}{du}
\DeclareMathOperator{\ev}{ev}
\DeclareMathOperator{\angles}{Ang}
\DeclareMathOperator{\FI}{FI}

\DeclareMathOperator{\FS}{FS}
\DeclareMathOperator{\Tree}{Tree}
\DeclareMathOperator{\Treecv}{\Tree(\N^*)^{\mathrm{conv}}}
\DeclareMathOperator{\Schtree}{T} 
\DeclareMathOperator{\Dend}{Dend}
\DeclareMathOperator{\Tridend}{Tridend}

\DeclareMathOperator{\flaten}{flat}

\DeclareMathOperator{\Int}{int}
\DeclareMathOperator{\B}{B}
\DeclareMathOperator{\BT}{\mathcal{BT}}

\DeclareMathOperator{\Borel}{\mathcal{B}}
\DeclareMathOperator{\Bint}{\mathfrak{B}}
\DeclareMathOperator{\Neg}{\mathcal{N}}
\DeclareMathOperator{\qshuffle}{\overline{\shuffle}}

\newcommand{\calA}{\mathcal{A}}

\newcommand{\calP}{\mathcal{P}}
\newcommand{\calS}{\Suum}
\newcommand{\calW}{\mathcal{W}}
\newcommand{\Dendcv}{\ensuremath{\Dend(\ensuremath{\{x,y\}})^{\rm conv}}}
\newcommand{\W}{\mathcal{W}_{\N^*}}
\newcommand{\Wcv}{\mathcal{W}^{\mathrm{conv}}_{\N^*}}
\newcommand{\Wxy}{\mathcal{W}_{\{x,y\}}}
\newcommand{\Wxycv}{\mathcal{W}^{\mathrm{conv}}_{\{x,y\}}}

\NewDocumentCommand{\circlecoloured}{m}{\begin{tikzpicture}
		\filldraw[color=#1] (0,0) circle (1pt);
\end{tikzpicture}}

\NewDocumentCommand{\labY}{m}{\raisebox{-0.4\height}{\begin{tikzpicture}[line cap=round,line join=round,>=triangle 45,x=0.4cm,y=0.4cm]
			\draw [line width=.5pt] (0.,0.)-- (0.,1.);
			\draw [line width=.5pt] (0.,1.)-- (-1.,2.);
			\draw [line width=.5pt] (0.,1.)-- (1.,2.);
			\draw[above] (0,1) node {\footnotesize{#1}}; 
\end{tikzpicture}}}

\NewDocumentCommand{\labbalaisg}{m m}{\raisebox{-0.3\height}{\begin{tikzpicture}[line cap=round,line join=round,>=triangle 45,x=0.4cm,y=0.4cm]
			\draw (0,0)-- (0,1);
			\draw (0,1)-- (-1,2);
			\draw (-0.5,1.5)-- (0,2);
			\draw (0,1)-- (1,2);
			\draw[above] (0,1) node {\footnotesize{#1}};
			\draw[above] (-0.5,1.5) node {\footnotesize{#2}};
\end{tikzpicture}}}

\NewDocumentCommand{\labbalaisd}{m m}{\raisebox{-0.3\height}{\begin{tikzpicture}[line cap=round,line join=round,>=triangle 45,x=0.4cm,y=0.4cm]
			\draw (0,0)-- (0,1);
			\draw (0,1)-- (-1,2);
			\draw (0.5,1.5)-- (0,2);
			\draw (0,1)-- (1,2);
			\draw[above] (0,1) node {\footnotesize{#1}};
			\draw[above] (0.5,1.5) node {\footnotesize{#2}};
\end{tikzpicture}}}

\NewDocumentCommand{\labbalais}{m m}{\raisebox{-0.3\height}{\begin{tikzpicture}[line cap=round,line join=round,>=triangle 45,x=0.4cm,y=0.4cm]
			\draw (0,0)-- (0,2);
			\draw (0,1)-- (-1,2);
			\draw (0,1)-- (1,2);
			\draw (0.4,1.75) node {\footnotesize{#1}};
			\draw (-0.4,1.75) node {\footnotesize{#2}};
\end{tikzpicture}}}

\NewDocumentCommand{\labbalaisbis}{m}{\raisebox{-0.3\height}{\begin{tikzpicture}[line cap=round,line join=round,>=triangle 45,x=0.4cm,y=0.4cm]
			\draw (0,0)-- (0,2);
			\draw (0,1)-- (-1,2);
			\draw (0,1)-- (1,2);
			\draw[right] (0,0.75) node {\footnotesize{#1}};
\end{tikzpicture}}}

\NewDocumentCommand{\labYY}{m m m }{\raisebox{-0.3\height}{\begin{tikzpicture}[line cap=round,line join=round,>=triangle 45,x=0.2cm,y=0.2cm]
			\draw (0,0)-- (0,1);
			\draw (0,1)-- (-2,3);
			\draw (-1.25,2.25)-- (-0.5,3);
			\draw (1.25,2.25)-- (0.5,3);
			\draw (0,1)-- (2,3);
			\draw[above] (0,1) node {\footnotesize{#1}};
			\draw[above] (-1.25,2.25) node {\footnotesize{#2}};
			\draw[above] (1.25,2.25) node {\footnotesize{#3}};
\end{tikzpicture}}}

\NewDocumentCommand{\labYYI}{m m m }{\raisebox{-0.3\height}{\begin{tikzpicture}[line cap=round,line join=round,>=triangle 45,x=0.25cm,y=0.25cm]
			\draw (0,0)-- (0,1);
			\draw (0,1)-- (-2,3);
			\draw (0,1) -- (0,2.5);
			\draw (-1.5,2.5)-- (-1,3);
			\draw (0,2.5)-- (0.5,3);
			\draw (0,2.5)-- (-0.5,3);
			\draw (0,1)-- (2,3);
			\draw[left] (0,1) node {\footnotesize{#1}};
			\draw[above] (-1.5,2.5) node {\footnotesize{#2}};
			\draw[above] (0,2.5) node {\footnotesize{#3}};
\end{tikzpicture}}}

\NewDocumentCommand{\alabelledtree}{m m m}{\raisebox{-0.4\height}{\begin{tikzpicture}[line cap=round,line join=round,>=triangle 45,x=0.3cm,y=0.3cm]
			\draw (0,0) -- (0,1);
			\draw (0,1) -- (-2,3);
			\draw (0,1) -- (2,3);
			\draw (-1,2) -- (-1,3);
			\draw (-1,2) -- (0,3);
			\draw (1.5,2.5) -- (1,3);
			\draw[above] (0,1) node{\footnotesize{#1}};
			\draw[left] (-1,2) node{\footnotesize{#2}};
			\draw[above] (1.5,2.5) node{\footnotesize{#3}};
\end{tikzpicture}}}

\newcommand{\peignedroitdec}[6]{\raisebox{-0.5\height}{\begin{tikzpicture}[line cap=round,line join=round,>=triangle 45,x=0.3cm,y=0.3cm]
			\draw (0,0)--(0,1);
			\draw (0,1)--(-1,2) node[left]{$#1$};
			\draw (0,1)--(6,7);
			\draw (2,3)--(1,4) node[left]{$#2$};
			\draw (3.,4.5) node[rotate=45]{$\cdots$};
			\draw (5,6)--(4,7) node[left]{$#3$};
			\draw (0,1) node[right]{\scriptsize{$#4$}};
			\draw (2,3) node[right]{\scriptsize{$#5$}};
			\draw (5,6) node[right]{\scriptsize{$#6$}};
\end{tikzpicture} }}
\newcommand{\peignegauchedec}[6]{\raisebox{-0.5\height}{\begin{tikzpicture}[line cap=round,line join=round,>=triangle 45,x=0.3cm,y=0.3cm]
			\draw (0,0)--(0,1);
			\draw (0,1)--(1,2) node[right]{$#1$};
			\draw (0,1)--(-6,7);
			\draw (-2,3)--(-1,4) node[right]{$#2$};
			\draw (-3.,4.5) node[rotate=135]{$\cdots$};
			\draw (-5,6)--(-4,7) node[right]{$#3$};
			\draw (0,1) node[left]{\scriptsize{$#4$}};
			\draw (-2,3) node[left]{\scriptsize{$#5$}};
			\draw (-5,6) node[left]{\scriptsize{$#6$}};
\end{tikzpicture}}}

\newcommand{\labYI}[4]{\raisebox{-0.3\height}{\begin{tikzpicture}[line cap=round,line join=round,>=triangle 45,x=0.3cm,y=0.3cm]
			\draw (0,0)-- (0,1);
			\draw (0,1)-- (-2,3);
			\draw (0,1)-- (2,3);
			\draw (-1.5,2.5)-- (-1,3);
			\draw (0.75,1.75)-- (-0.5,3);
			\draw (1.5,2.5)-- (1,3);
			\draw[above] (0,1) node {\footnotesize{#1}};
			\draw[above] (-1.5,2.5) node {\footnotesize{#2}};
			\draw[above] (0.75,1.75) node {\footnotesize{#3}};
			\draw[above] (1.5,2.5) node {\footnotesize{#4}};
\end{tikzpicture}}}

\newcommand{\labIY}[4]{\raisebox{-0.3\height}{\begin{tikzpicture}[x=0.3cm,y=0.3cm]
			\draw (0,0) -- (0,1);
			\draw (0,1) -- (-2,3);
			\draw (0,1) -- (2,3);
			\draw (0,2.5) -- (0.5,3);
			\draw (0.75,1.75) -- (-0.5,3);
			\draw (-1.5,2.5) -- (-1,3);
			\draw (0,1) node[above]{\footnotesize{$#1$}};
			\draw (0.75,1.75) node[above]{\footnotesize{$#2$}};
			\draw (-1.5,2.5) node[above]{\footnotesize{$#3$}};
			\draw (0,2.5) node[above]{\footnotesize{$#4$}};
	\end{tikzpicture}}}
\newcommand{\labIpY}[3]{
\raisebox{-0.3\height}{\begin{tikzpicture}[x=0.3cm,y=0.3cm]
		\draw (0,0) -- (0,1);
		\draw (0,1) -- (-2,3);
		\draw (0,1) -- (2,3);
		\draw (1,2) -- (1,3);
		\draw (1,2) -- (0,3);
		\draw (-1.5,2.5) -- (-1,3);
		\draw (1,2) node[right]{\footnotesize{$#2$}};
		\draw (-1.5,2.5) node[above]{\footnotesize{$#3$}};
		\draw (0,1) node[above]{\footnotesize{$#1$}};
\end{tikzpicture}}
}

\newcommand{\labYgI}[4]{
\raisebox{-0.3\height}{\begin{tikzpicture}[x=0.3cm,y=0.3cm]
		\draw (0,0) -- (0,1);
		\draw (0,1) -- (-2,3);
		\draw (0,1) -- (2,3);
		\draw (-0.75,1.75) -- (0.5,3);
		\draw (-1.5,2.5) -- (-1,3);
		\draw (0,2.5) -- (-0.5,3);
		\draw (0,1) node[above]{\footnotesize{$#1$}};
		\draw (-0.75,1.75) node[above]{\footnotesize{$#2$}};
		\draw (-1.5,2.5) node[above]{\footnotesize{$#3$}};
		\draw (0,2.5) node[above]{\footnotesize{$#4$}};
\end{tikzpicture}}
}

\providecommand{\keywords}[1]{\textbf{\textit{Keywords.}} #1}
\providecommand{\AMSclass}[1]{\textbf{\textit{AMS classification.}} #1}

	\title{{\bfseries Tridendriform and dendriform Zeta Values \\ from \Sch{} trees }}
	\author[1,2]{Pierre Catoire \thanks{Contact: \href{mailto: catoire_research@proton.me;}{catoire\textunderscore{}research@proton.me
		}}} 	
	\author[3]{Pierre Clavier}
	\author[3]{Douglas Modesto da Fraga Candido}
		\affil[1]{Université d'Artois, UR 2462, Laboratoire de Mathématiques de Lens (LML), F-62300 Lens,
			France
			}
		\affil[2]{Institut Montpellierain Alexander Grothendieck, Université de Montpellier, Case
			Courrier 051 - Place Eugène Bataillon, 34095 Montpellier Cedex 5, France}
		\affil[3]{Université de Haute Alsace, IRIMAS, 12 rue des Frères Lumière, MULHOUSE Cedex 68 093,
			France
			}

\begin{document}	
	\maketitle
	\begin{abstract}
		To build new generalisations of \MZVs{}, we define new spaces of formal series and formal integrals. We show that they are tridendriform and dendriform algebras. This allows us to reinterpret the fact that \MZVs{} are algebra morphisms for shuffles of words in terms of finer tridendriform and dendriform structures. 
		Applying universal properties of \Sch{} trees we obtain generalisations of \MZVs{} that are algebra morphisms for associative products.
		Hence we find new properties of \AZVs{} and state how this enables the computation of some \SZVs{}.
	\end{abstract}
	
	\AMSclass{16W99, 16S10, 11E45, 05C05}
	
	\tableofcontents
	
	\section*{Introduction}
	
	\addcontentsline{toc}{section}{Introduction}

	\subsection*{Notations}
	
	\addcontentsline{toc}{subsection}{Notations}
	
	For any set $\Omega$, we will denote:
	\begin{itemize}
		\item for any set $X$, we denote by $\K X$ the vector space whose basis is the set $X$.
		\item $\Omega^{\star}$ the set of finite words on the alphabet $\Omega$;
		\item $\calW_\Omega$ the unique $\Q$-vector space whose basis is given by elements of $\Omega^{\star}$;
		\item We write $\ew\in\calW_\Omega$ the empty word;
		\item For $\Omega=\N^*$, $\Wcv$ is the sub-vector space of $\W$ spanned by the empty word and words that do not start by $1$;
		\item For $\Omega=\{x,y\}$, $\Wxycv$ is the sub-vector space of $\Wxy$ spanned by the empty word and words that start with $x$ and end with $y$;
		\item for any $n\in\N^*$, we denote $\calW_{\Omega,n}$ the sub-vector space generated by all the words of length exactly $n$;
	\end{itemize}

	\subsection*{\MZVs{} and their generalisations}
	
	\addcontentsline{toc}{subsection}{\MZVs{} and their generalisations}
	
		The \MZVs{}\footnote{also called polyzeta or other names by various authors, but we use here what seems to be now the most widespread designation.} where introduced by Leonard Euler in~1775~\cite{Eu1796} and have been rediscovered many times since then. The modern interest of this topic was sparked by the work of~Ecalle\cite{Ecalle81b} and the systematic study of \MZVs{} started with the works of Hoffman~\cite{Ho92} and Zagier~\cite{Za94}. There are many open conjectures and known results regarding the general theory of \MZVs{} and it is not the purpose of this introduction to present them all. We refer the interested readers to, for example,~\cite{dupont2019valeurs} for a presentation of this exciting field of research.
		
		Let us instead start by giving the relevant definitions of \MZVs{}:
		\begin{defi*}[\MZV{} function] 
		We define the following map:
		\begin{equation} \label{def:MZVs}
		\zeta:\left\lbrace\begin{array}{rcl}
			\Wcv &\longrightarrow & \R, \\
			w_1\dots w_k & \longmapsto & \displaystyle\sum_{1\leq n_k<\dots <n_2 < n_1} \prod_{i=1}^k n_i^{-w_i}, \\
			\ew  & \longmapsto & 1.
		\end{array} \right.
		\end{equation}
	\end{defi*}
	 We then have another representation of the numbers in the image of $\zeta$, given by iterated integrals.
	\begin{defi*}[Integral representations]
		We introduce the following map:
		\begin{equation} \label{Prop:integral_representation_classic}
		\zeta_{\Int}:\left\lbrace\begin{array}{rcl}
			\Wxycv & \longrightarrow & \R, \\
			\omega_1\dots\omega_n & \longmapsto & \displaystyle\int_{0< u_n < \cdots < u_1 < 1} \prod_{i=1}^{n} g_{\omega_i}(u_i)\du_i \\
			\ew  & \longmapsto & 1.
		\end{array} \right.
		\end{equation}
		where:
		\begin{align*}
			g_x:\left\lbrace\begin{array}{rcl}
				\interoo{0 1} & \rightarrow & \R, \\
				t & \mapsto & \frac{1}{t}, 
			\end{array} \right.  && g_y:\left\lbrace\begin{array}{rcl}
				\interoo{0 1} & \rightarrow & \R, \\
				t & \mapsto & \frac{1}{1-t}.
			\end{array} \right. 
		\end{align*}
	\end{defi*}
	A result usually attributed to Kontsevitch but published first in~\cite{Za94}\footnote{although, as pointed out in~\cite{Manchon_16}, a remark of~\cite{Ecalle81b} might refer to this fact.} is that these iterated series and integrals are different ways of writing the same real numbers. These two representations are linked by the following map:
	\begin{defi*}
		The \emph{binarization map} is the unique linear map defined by:
		\[
		\mathfrak{s}:\left\lbrace\begin{array}{rcl}
			\W & \rightarrow & \Wxy, \\
			n_1\dots n_k & \mapsto & x^{n_1-1}yx^{n_2-1}y \dots x^{n_k-1}y.
		\end{array} \right.
		\]
	\end{defi*}
	It restricts to a bijection between $\Wcv$ and $\Wxycv$. This map is a bridge between the two representations of \MZVs:
	\begin{Prop*}[Kontsevitch's relation~\cite{Za94}]
		We have $\zeta=\zeta_{\Int}\circ \mathfrak{s}|_{\Wcv}.$
	\end{Prop*}
	The spaces $\Wxy$ and $\W$ have the structure of commutative algebras\footnote{actually, of Hopf algebras} respectively with the \emph{shuffle} and \emph{quasi-shuffle} as products respectively denoted $\shuffle$ and $\qshuffle$ \footnote{provided $\Omega$ as a commutative semigroup structure, both algebra structures over $\calW_\Omega$ given by $\shuffle$ or $\qshuffle$ are isomorphic~\cite{Ho00}}. These products will play important roles in this paper and are defined in details in
	definitions~\ref{def:quasi_shuffle_words} and~\ref{def:shuffle_words} below. Then, a crucial but well-known result is
	\begin{thm*}
		The maps $\zeta$ and $\zeta_{\Int}$ are respectively algebra morphisms for the quasi-shuffle and the shuffle product: for any two elements $u$ and $v$ of $\Wcv$ and any two elements $u'$ and $v'$ of $\Wxycv$, we have
		\begin{equation} \label{eq:shuffle_stuffle_zeta}
			\zeta(u\cshuffle v)=\zeta(u)\cdot \zeta(v),\qquad \zeta_{\Int}(u'\shuffle v')=\zeta_{\Int}(u')\cdot \zeta_{\Int}(v'). 
		\end{equation}
	\end{thm*}
	The quasi-shuffle version seems to be due to Hoffmann~\cite{Ho00} but some properties of MZVs appeared earlier in~\cite{Ecalle81,Ecalle81b}. See for example~\cite{Waldschmidt} for a more detailed presentation of these results.
	
	\medskip
	
	Since, \MZVs{} have been generalised in many directions. Again, it is not the purpose of this introduction to be a review, so let us just mention the generalisations that will play a role in this paper. Relevant definitions will be introduced within the main text of this paper.
	\begin{itemize}
		\item Since words can be seen as rooted trees without branching, a natural generalisation of \MZVs{} is to write them as maps with (convergent) rooted trees and forests as domain. This was introduced in~\cite{Manchon_16} and independently in~\cite{Ya20}. A systematic study of these \emph{arborified zeta values} (see definition~\ref{def:azvs}) was started in~\cite{clavier2020double} and further developed in~\cite{clavier2024generalisations}. The iterated series version of these arborified zeta values is given below in definition~\ref{def:azvs}, while the iterated integrals version in written in definition~\ref{def:azvs_integrals}.
		\item Another generalisation of \MZVs{} that will play a role here are the \emph{\SZVs{}} (definition~\ref{def:shintani}). The (single valued) \Shin{} zeta were introduced in~\cite{shintani1976evaluation} and their multiple counterpart in~\cite{cassou1979valeurs,2013multidimensional}. These were further studied in~\cite{matsumoto2003mordell,lopez2023milnor}.
		\item Notice that some results of~\cite{chapoton2022zinbiel} are close to ours (for the dendriform part), since it studies the Zinbiel structure behind \MZVs. The paper~\cite{chapoton2022zinbiel} belongs to the theory of \emph{motivic zeta values}, which is thus related to our results. For this approach to \MZV{} see for example~\cite{gil2017multiple}.
	\end{itemize}
	
	\subsection*{Scope and main results}
	
	\addcontentsline{toc}{subsection}{Scope and main results}
	
	The first goal of this paper is to use the freeness property of a (tri)dendriform structures of rooted trees to build a new generalisation of \MZVs{} to rooted trees. Dendriform algebras were introduced by Loday~\cite{loday1995algebres,loday1993version} and their free objects were described in~\cite{loday2006structure}. Tridendriform algebras were introduced in~\cite{loday2002trialgebras} and their free objects were built in~\cite{zhang2020free,Catoire_23}. Dendriform and tridendriform algebras have been since then quite an active field of research and applied to many areas of mathematics and even physics. We refer the reader to the introduction of~\cite{zhang2020free} for an overview of these applications.
	
	For this, we introduce spaces of formal series in definition~\ref{defi:formal_series} and of formal integrals in definition~\ref{defi:formal_integrals}. We then show that some subspaces of formal series and formal integrals admit respectively a natural tridendriform algebra structure and dendriform algebra structure. Applying the universal properties for these spaces of formal series and formal integrals we obtain generalisations of \MZVs{} to two distinct spaces of Schroeder trees in definitions~\ref{def:triden_zeta} (\emph{tridendriform zeta values}) and~\ref{def:dend_zeta} (\emph{dendriform zeta values}). The tridendriform version (definition~\ref{def:triden_zeta}) is a generalisation of equation~\eqref{def:MZVs} while the dendriform version (definition~\ref{def:dend_zeta}) is a generalisation of equation~\eqref{Prop:integral_representation_classic}. Thus definitions~\ref{def:triden_zeta} and~\ref{def:dend_zeta} completely fulfill our first objective.
	
	\medskip
	
	Our second objective is to study the properties of the aforementioned tridendriform and dendriform zeta values, both algebraic and number-theoretic. Regarding the algebraic part, from their definitions we directly have that the dendriform and tridendriform zeta values are algebra morphisms for the quasi-shuffle and shuffle of Schroeder trees respectively. These products are associative and not commutative generalisations of the quasi-shuffle and shuffle products of words of equation~\eqref{eq:shuffle_stuffle_zeta}.
	
	For the number-theoretic aspects of this question, we show that tridendriform and dendriform zeta values are closely related to arborified zeta values, both in the iterated series version (corollary~\ref{coro:relation_zetas}) and in the iterated integral one (corollary~\ref{coro:relation_zetas_int}). This in turn implies that dendriform and tridendriform zeta values can be written as linear combinations of \MZVs{} with rational coefficients, in a completely explicit way. Thus our second goal is also reached.
	
	\medskip
	
	The third and final goal of this paper is to apply the unravelled tridendriform and dendriform structures to obtain new results regarding \MZVs{} and some of their generalisations. First, we show that the map $\zeta:\Wcv\longrightarrow\R$ and $\zeta_{\Int}:\Wxycv\longrightarrow\R$ of equations~\eqref{def:MZVs} and~\eqref{Prop:integral_representation_classic} can be decomposed as
	\begin{equation*}
	 \zeta=\ev\circ\zeta_{\rm FS},\qquad\zeta_{\rm int}=\ev\circ\zeta_{\rm FI}.
	\end{equation*}
	In these decompositions, $\ev$ are two evaluation maps from spaces of formal series and integrals to real numbers. However, strikingly, $\zeta_{\rm FS}$ and $\zeta_{\rm FI}$ are morphisms of tridendriform and dendriform algebras respectively. In other words, evaluating the \MZVs{} makes us lose some of their algebraic properties. 
	
	We also show that arborified zeta values in their iterated series (resp. integral) form are an algebra morphism for a quasi-shuffle (resp. shuffle) product of decorated Schroeder trees. These results are stated in corollaries~\ref{coro:AZV_alg_morph} and~\ref{cor:azvs_dendri_alg}. Notice that these products are associative and not commutative. As such, they are more natural generalisations to (Schroeder) trees of the quasi-shuffle and shuffle products of words than the ones previously built in~\cite{clavier2020double,clavier2024generalisations}.
	
	Our last applications is for \SZVs{}. We show in theorem~\ref{thm:dend_zeta_series} that dendriform zeta values can be written as iterated series. This series are \Shin{} zetas, so the previous results implies that a large family of \SZVs{} can be written as \MZVs{} with rational coefficients. 
	
	\subsection*{Content}
	
	\addcontentsline{toc}{subsection}{Content}

	The paper is composed of four parts:
	\begin{enumerate}
		\item The first one composed of section~\ref{sec:one} presents the objects underlining most of this paper: tridendriform and dendriform algebras. They are introduced in definitions~\ref{defi:tridend} and~\ref{defi:dend}. We then describe the classical dendriform algebra spanned by words written over a set, and the tridendriform algebra words written over $\N^*$. These structures serve as examples but are also important in the reminder of this work. 
		
		\item  The second part deals with tridendriform structures and generalisations of \MZVs{} as iterated series (equation~\eqref{def:MZVs}). It goes from section~\ref{sec:tridend} where we introduce \emph{formal series} (definition~\ref{defi:formal_series}) to section~\ref{sec:tridend_to_usual} establishing a bridge with usual \MZVs{} (corollary~\ref{coro:relation_zetas}). 
		
		In section~\ref{sec:tridend}, some of the standard properties of usual series are shown to hold in this framework in proposition~\ref{prop:formal_series_properties}. We then introduce in definition~\ref{def:formal_series_MZVs} a subspace of formal series together with an evaluation map (definition~\ref{defi:ev_map_tri}) which sends some  convergent (in some sense) formal series to real numbers. This subspace of formal series is then shown to have the structure of a tridendriform algebra (proposition~\ref{Prop:tridend_Suum}). We then introduce a formal series version of \MZVs{} (definition~\ref{def:zeta_FS}) which, when composed with the evaluation map gives back the usual \MZVs{} of equation~\eqref{def:MZVs}. We further show in proposition~\ref{Lem:zeta_morph_tridend} that this formal series version of \MZVs{} forms a morphism of tridendriform algebras.
		
		We start section~\ref{sec:three} by recalling some definitions, namely of Schroeder trees (definition~\ref{defi:schroeder}) and their tridendriform structures (definition~\ref{def:quasiaction} and theorem~\ref{thm:produit}). Their universal property in the category of tridendriform algebras is recalled in theorem~\ref{thm:univ_prop}. Applying this universal property to the tridendriform structure of formal series, we obtain a map dubbed \emph{tridendriform zeta values} in definition~\ref{def:triden_zeta}, which is by construction a morphism of tridendriform algebras, and in particular an algebra morphism for the quasi-shuffle of \Sch{} trees (equation~\eqref{eq:tridend_zeta_morphism}).
		
		Section~\ref{sec:tridend_to_usual} is the last of the aforementioned second part. It starts with a definition of Schroeder trees decorated on their vertices (definition~\ref{defi:schroeder}) and of a map $\iota$ relating these decorated \Sch{} trees to the ones previously used (definition~\ref{defi:iota}). Thanks to this map, this space of decorated \Sch{} trees inherits a tridendriform algebra structure from the one of usual \Sch{} trees (proposition~\ref{prop:tree_tridend}). We then recall the definition of arborified zeta values as iterated series (definition~\ref{def:azvs}), of the flattening map (definition~\ref{defi:flat_map}) and state in theorem~\ref{thm:knowm_AZV} some known properties of arborified zeta values. We show in theorem~\ref{thm:zeta_tridend} that these various spaces and maps nicely fit together. Specialising this theorem to \emph{convergent Schroeder trees} (definition~\ref{def:Schroeder_tree_convergent}) and applying the evaluation map mentioned above we obtain that tridendriform zeta values are related to arborified zeta values and to the usual \MZVs{} (corollary~\ref{coro:relation_zetas}). This implies a new algebraic property for arborified zeta values (corollary~\ref{coro:AZV_alg_morph}).
		
		\item 	The third part follows roughly the same structure than the first one, but uses dendriform structures rather than tridendriform ones. It starts in section~\ref{sec:dendri} by the introduction of \emph{formal integrals} (definition~\ref{defi:formal_integrals}) and ends with the connection with \MZVs{} in section~\ref{sec:dendri_to_usual} (corollary~\ref{coro:relation_zetas_int}).
		
		Section~\ref{sec:dendri} introduces formal integrals and shows some usual properties of integrals are still true for their formal version. We then look at a particular subspace of formal integrals, namely formal Chen integrals (definition~\ref{def:formal_Chen_int}). It allows us to find back the usual integral representations of \MZVs{} via a formal version of \MZVs{} (definition~\ref{def:formal_MZV}). Then, the evaluation map (definition~\ref{defi:ev_map_dend}) sends convergent words to real numbers. The subspace of Chen integrals together with the evaluation map gives back the usual integral \MZVs{} of equation~\eqref{Prop:integral_representation_classic}. Moreover, we show that the space of Chen's formal integrals has a dendriform structure (proposition~\ref{prop:chen_dend_struct}) and the formal version of \MZVs{} is a morphism of dendriform algebras (proposition~\ref{prop:formal_MZV_dend_map}).
		
		Section~\ref{sec:six} states how the free dendriform algebra is built over decorated binary trees (definitions~\ref{def:binary_tree}) and its universal property (theorem~\ref{thm:univ_prop_dend}). We  get a \emph{dendriform zeta values} applying the universal property of the free dendriform algebra. In particular, we obtain  a dendriform morphism for binary trees decorated by $\{x,y\}$ (equation~\eqref{eq:formal_integrals_building}).
		
		In section~\ref{sec:dendri_to_usual}, the last of this third part, we recall the definition of integral \AZVs{} (definition~\ref{def:azvs_integrals}) as iterated integrals, the flattening map (definition~\ref{def:flat_map_dend}) and state in theorem~\ref{thm:known_AZVs_dend} some known results about integral \AZVs{}. We show in theorem~\ref{thm:dend_diag} that all those maps behave well with each other. Specialising this theorem to convergent binary trees (first point of definition~\ref{def:azvs_integrals}) and applying the evaluation map of formal integrals we get that dendriform zeta values are related to usual \MZVs{} (corollary~\ref{coro:relation_zetas_int}).

		\item The last part is composed of section~\ref{sec:Shintani} dedicated to the study of \SZVs{} (definition~\ref{def:shintani}). We obtain that dendriform zeta values are \SZVs{} (theorem~\ref{thm:dend_zeta_series}).
	\end{enumerate}
	
	\subsection*{Acknowledgement}
	
	 The three authors would like to thank the IRIMAS for inviting the first author to come to the UHA for a short research stay where this project was started and ANR Carplo for providing financial support. The first author would like to thank Laboratoire de Mathématiques Pures et Appliquées of ULCO, the Laboratoire de Mathématiques de Lens of université d'Artois for financing a part of this project and Institut Montpellierain Alexander Grothendieck from université de Montpellier for welcoming me.

\section{Dendriform and Tridendriform algebras} \label{sec:one}

	\subsection{Description of algebraic structures}
	
	Let us start by defining the central structures of this paper, namely tridendriform and dendriform algebras. These definitions are taken from~\cite{burgunder2010tridendriform,loday2002trialgebras,foissy2007bidendriform,ronco2002eulerian}. 
	\begin{defi}\label{defi:tridend}
		Let $A$ be an vector space  endowed with three bilinear operations $\prec,\cdot,\succ$.  We say that $(A,\succ,\prec,\cdot)$ is a \emph{tridendriform algebra} if for all $(x,y,z)\in A^3$:
		\begin{align}
			(x\prec y)\prec z&=x\prec(y*z), \label{eq:tri1}\\
			(x\succ y)\prec z&=x\succ(y\prec z), \label{eq:tri2} \\
			(x* y)\succ z&=x\succ(y\succ z),  \label{eq:tri3} \\
			(x\succ y)\cdot z&=x\succ(y\cdot z),	 \label{eq:tri4} \\
			(x\prec y)\cdot z&=x\cdot(y\succ z), \label{eq:tri5} \\
			(x\cdot y)\prec z&=x\cdot(y\prec z), \label{eq:tri6} \\
			(x\cdot y)\cdot z&=x\cdot (y\cdot z), \label{eq:tri7}
		\end{align}
		where for all $x,y\in A$, we set $x*y\coloneqq x\prec y + x\cdot y +x\succ y$. We respectively call $\prec, \succ, \cdot$ the \emph{left, right and middle products}.
	\end{defi}
	A dendriform algebra can be seen as a tridendriform algebra with a vanishing middle product.
	\begin{defi}\label{defi:dend}
		We say that a tridendriform algebra $(A,\prec,\succ,\cdot)$ is a \emph{dendriform algebra} if $\cdot=0$. Only three relations remains:
		\begin{align}
			(x\prec y)\prec z&=x\prec(y\star z), \label{eq:tri1}\\
			(x\succ y)\prec z&=x\succ(y\prec z), \label{eq:tri2} \\
			(x\star y)\succ z&=x\succ(y\succ z),  \label{eq:tri3} 
		\end{align}
		where $\star\coloneqq\prec+\succ$.
	\end{defi}
	The following simple proposition indicates that (tri)dendriform structures are endowed with a classical associative algebra structure. 
	\begin{Prop}
	 Let $(A,\succ,\prec,\cdot)$ (resp $(A,\succ,\prec)$) be a tridendriform algebra (resp. a dendriform algebra). Then the product $*$ on $A$ defined by $x*y\coloneqq x\prec y + x\cdot y +x\succ y$ (resp. $x\star y\coloneqq x\prec y+x\succ y$) is associative.
	\end{Prop}
	\begin{proof}
		One just needs to write the associator for $*$ (or $\star$) and develop it with the definition using $\prec,\succ$ and $\cdot$.
	\end{proof}
	The definition of tridendriform algebra does not allow the existence of an unit. Hence we need to add it by hand:
	\begin{defi}
		Let $(A,\succ,\prec,\cdot)$  be a tridendriform algebra. We expand the definitions of the products over the space $(\K\otimes A)\oplus (A\otimes \K) \oplus (A\otimes A)$ putting :
		\begin{align}
			&\forall a\in A,& 1\prec a \coloneqq0=a\succ 1&,& a\prec 1\coloneqq a=1\succ a& &\text{  and  }& &a\cdot 1\coloneqq0=1\cdot a. \label{eq:augm}
		\end{align}
		We put $\overline{A}\coloneqq A\oplus \K\cdot 1.$
		We call $\left(\overline{A}, \prec,\cdot,\succ, 1\right)$ the \emph{augmented tridendriform algebra}. 
	\end{defi}
	
	The notion of tridendriform (and dendriform putting $\cdot=0$) comes with its notion of morphisms and subalgebras:
		\begin{defi}
			A \emph{tridendriform morphism} of augmented tridendriform algebra between ${(\overline{A},\prec,\succ, \cdot)}$ and ${(\overline{B},\prec,\succ,\cdot)}$ is a linear map $f:A\rightarrow B$ such that for any $x,y\in A$:
			\begin{align*}
				f(x\prec y)=f(x)\prec f(y), && f(x\succ y)=f(x)\succ f(y), && f(x\cdot y)=f(x)\cdot f(y),
			\end{align*}
			and $f(1)=1$.
		\end{defi}
		\begin{defi}[tridendriform subalgebra]
			Let $(A,\prec,\succ,\cdot)$ be an augmented tridendriform algebra. Let $B\subseteq A$, we say it is a \emph{tridendriform subalgebra} if it is a subalgebra for $(A,*,1)$ and it is closed for the operations $\prec,\succ$ and $\cdot$.
		\end{defi}
	
	\begin{Rq}
		For sake of simplicity, when we talk about tridendriform or dendriform algebra later, it will mean it is an augmented tridendriform or dendriform algebra.
	\end{Rq}
	
\subsection{Dendriform and tridendriform structures for words}
    
We recall classical examples of (tri)dendriform structures on words. These examples are crucial in the theory of \MZVs{} (MZVs), see for example~\cite{Waldschmidt,Ho00}.
	\begin{defi}[shuffle product for words] \label{def:shuffle_words}
			Let $\Omega$ be a set.
		Let $u=u_1\dots u_n$ and $v=v_1\dots v_k$ be two words in $\Omega^{\star}$. We define inductively over $n+k$ the \emph{shuffle product} by $u\shuffle \emptyset=\emptyset\shuffle u=u$ for all $u$ in $\calW_\Omega$ and:
		\[
		u\shuffle v\coloneqq u_1(u_2\dots u_n \shuffle v)+ v_1(u\shuffle v_2\dots v_k).
		\]
		We extend by linearity this product to a linear map $\shuffle:\calW_\Omega^{\otimes 2} \rightarrow \calW_\Omega$.
		We can split this shuffle product into two smaller products $\succ$ and $\prec$ defined inductively by:
		\begin{align*}
			u\prec v = u_1(u_2\dots u_n \shuffle v), &&
			u\succ v=v_1(u\shuffle v_2\dots v_k).
		\end{align*}
	\end{defi}
		\begin{Eg}
		Let $xy$ be an element of 
			$\calW_{\{x,y\}}$, then the shuffle product with itself is:
		\begin{align*}
			xy \shuffle xy&= 4~xxyy + 2~xyxy. 
		\end{align*}
	\end{Eg}
	
	\begin{defi}[quasi-shuffle product for words]  \label{def:quasi_shuffle_words} 
		Let $u=u_1\dots u_n$ and $v=v_1\dots v_k$ be two words in $\W$. We define inductively over $k$ and $l$ the \emph{quasi-shuffle product} by $u\qshuffle \emptyset=\emptyset\qshuffle u=u$ for all $u$ in $\W$ and:
		\[
		u\qshuffle v= u_1(u_2\dots u_n \qshuffle v)+ v_1(u\qshuffle v_2\dots v_k)+(u_1+v_1)(u_2\dots u_n \qshuffle v_2\dots v_k).
		\]
		We can split this shuffle product into three smaller products $\prec,\succ$ and $\cdot$ defined inductively by:
		\begin{align*}
			u\prec v = u_1(u_2\dots u_n \qshuffle v), &&
			u\succ v=v_1(u\qshuffle v_2\dots v_k), &&
			u \cdot v=(u_1+v_1)(u_2\dots u_n \qshuffle v_2\dots v_k).
		\end{align*}
	\end{defi}	
	\begin{Rq}
		Note that with these definitions neither $\ew \prec v$, $u \succ \ew$ nor $\ew\cdot v$ and $u\cdot\ew$ are defined, for $\succ$ and $\prec$ the left and right parts of the shuffle or quasi-shuffle product. For more details see~\cite{Catoire_23}.
	\end{Rq}
	
	\begin{Eg}[for the quasi-shuffle product]
		Consider $\mathbf{1}~\mathbf{2}$ and $\mathbf{3}~\mathbf{2}$ two elements of $\W$, the quasi-shuffle product of those two objects is:
		\begin{align*}
			\mathbf{1}~\mathbf{2} \qshuffle\mathbf{3}~\mathbf{2}
			&=\fourletters{1}{2}{3}{2}+\fourletters{1}{3}{2}{2}+\fourletters{1}{3}{2}{2}+\threeletters{1}{5}{2}+\threeletters{1}{3}{4} \\
			&+\fourletters{3}{2}{1}{2}+\fourletters{3}{1}{2}{2}+\fourletters{3}{1}{2}{2}+\threeletters{3}{3}{2}+\threeletters{3}{1}{4} +\threeletters{4}{2}{2} +\threeletters{4}{2}{2}+\twoletters{4}{4}.
		\end{align*}
	\end{Eg}
	\begin{Rq}
		The construction of tridendriform algebras is here described on the monoid $(\N^*,+)$. One could do this for any monoid $(M,\cdot)$ where the letters of our words are elements of $M$. We will not need this level of generality here and omit it for the sake of simplicity.
	\end{Rq}
	
	Shuffle and quasi-shuffle, with the decompositions presented above, are classical examples of dendriform and tridendriform algebras~\cite{Ebrahimi_Fard_2018}.
	\begin{Prop}
	 Let $\Omega$ be a set. Then $(\calW_\Omega,\prec,\succ)$ (resp. $\left(\W,\prec,\cdot,\succ,\ew\right)$) is a dendriform algebra (resp. an augmented tridendriform algebra).
	\end{Prop}
	
	We know from~\cite{Catoire_23} that the product of any tridendriform algebra is described by the following specific permutations that we now introduce.
	\begin{defi}\label{defi:qsh_sh2}
	\begin{itemize}
	 \item (Quasi-shuffles/sticky shuffles) Let $k,l\in\N\setminus \{0\}$. A \emph{$(k,l)$-\qsh{}} is a surjective map $\sigma:\IEM{1}{k+l}\twoheadrightarrow\IEM{1}{n}$ such that:
		\[
		\sigma(1)<\cdots<\sigma(k) \text{  and  } \sigma(k+1)<\cdots<\sigma(k+l).
		\]
		We write $\QSh(k,l)$ the set of $(k,l)$-\qsh{}.
		\item (Shuffle) Let $k,l\in\N\setminus \{0\}$. A \emph{$(k,l)$-\sh{}} is an element from $S_{k+l}$ such that:
		\[
		\sigma(1)<\cdots<\sigma(k) \text{  and  } \sigma(k+1)<\cdots<\sigma(k+l).
		\]
		We write $\Sh(k,l)$ the set of $(k,l)$-\sh{}.
	\end{itemize}
	\end{defi}

	Thank to these objects, we can describe the two shuffle products in terms or shuffles/quasi-shuffles (see~\cite{sweedler1969hopf} for the shuffle product and~\cite{GPZ13} for the quasi-shuffle):
	\begin{Lemme}
		Let $\Omega$ be a set. Consider $w=w_1\dots w_k$ and $w'=w_{k+1}\dots w_{k+l}$ two elements either of $\Omega^{\star}$ or $\W$.
		The shuffle product of two elements $w$ and $w'$ of $\Omega^{\star}$ can be described as:
		\[
		w\shuffle w'=\sum_{\sigma\in\Sh(k,l)} w_{\sigma^{-1}(\{1\})}\dots w_{\sigma^{-1}(\{\max(\sigma)\})}.
		\]
		The quasi-shuffle product of two elements $w$ and $w'$ of $\W$ can be described as:
		\[
		w\qshuffle w'=\sum_{\sigma\in\QSh(k,l)} w_{\sigma^{-1}(\{1\})}\dots w_{\sigma^{-1}(\{\max(\sigma)\})},
		\]
		where $w_{\{i,j\}}\coloneqq w_i+w_j$
	\end{Lemme}
	In particular, the pieces of each product corresponds to a condition over the first letter of the words appearing in the shuffle: 
	\begin{Lemme}
		Let $(k,l)\in\N^2,w=w_1\dots w_k$ and $w'=w_{k+1}\dots w_{k+l}.$ A non-inductive formula for the products of definition~\ref{def:quasi_shuffle_words} is given below:
		\begin{gather*}
			w\prec w'\coloneqq\sum_{\substack{\sigma\in\QSh(k,l) \\ \sigma^{-1}(\{1\})=\{1\}}} w_{\sigma^{-1}(\{1\})}\dots w_{\sigma^{-1}(\{\max(\sigma)\})}, \\
			 w\succ w'\coloneqq\sum_{\substack{\sigma\in\QSh(k,l) \\ \sigma^{-1}(\{1\})=\{k+1\}}} w_{\sigma^{-1}(\{1\})}\dots w_{\sigma^{-1}(\{\max(\sigma)\})}, \\
			w\cdot w'\coloneqq\sum_{\substack{\sigma\in\QSh(k,l) \\ \sigma^{-1}(\{1\})=\{1,k+1\}}} w_{\sigma^{-1}(\{1\})}\dots w_{\sigma^{-1}(\{\max(\sigma)\})}.
		\end{gather*}
	\end{Lemme}
	
	In the rest of the paper, we will often use the gradations of words by their lengths that we now recall.
	\begin{defi}
	 Let $\Omega$ be a set, and $w=w_1\dots w_k\in\calW_\Omega$ with $k\geq 1$. We set $l(w)\coloneqq k$ the \emph{length} of $w$. We also set $l(\emptyset)=0$. We write $\calW_{\Omega,k}$ the vector space generated by words of length $k$ written in the alphabet $\Omega$.
	\end{defi}

	In our applications, we will focus on the cases $\Omega=\{x,y\}$ and $\W$. Let us introduce some objects that are relevant in those situations.
	
	\begin{defi} \label{def:conv_words}
		We define $\Wcv$  the sub-vector space of $\W$ generated by the empty word and words whose first letters are an element of $\N^*$ different from $1$.
		
		Moreover, we introduce $\Wxycv$ the sub-vector space of $\Wxy$ such that its basis is given by the empty word and by elements that can be written as $x\Wxy y$.
	\end{defi}
		\begin{Rq}
			The structure $(\Wcv,\prec,\succ,\cdot)$ is a tridendriform subalgebra of $\W$. Moreover, $(\Wxycv,\prec,\succ)$ is a dendriform subalgebra of $\Wxy$.
		\end{Rq}
	
	\section{The tridendriform formal series algebra}\label{sec:tridend}
		
	\subsection{Formal series}
	
	In this subsection, the tool to introduce a tridendriform structure is formal series.
	Our construction will mainly involve the monoid of positive integers $(\N^*,+)$. But tridendriform structures can be usually defined over $\calW_{\Omega}$ for any monoid $(\Omega,\cdot)$.
		
	\begin{Not}
		We denote $(\N^*\rightarrow \R)^{\N^*}$ the set of maps from $\N^*$ to the set of functions from $\N^*$ to $\R$.
		It as a natural structure of vector space over $\R$. 
	\end{Not}
	
	\begin{defi}[Formal series over $\R$]\label{defi:formal_series}
		Consider the commutative monoid structure $(\N^*,+)$. Let $w\in\W$ and a map $F:\N^* \rightarrow (\N^* \rightarrow \R)$. We denote for any $n\in\N^*, f_n\coloneqq F(n)$. Then for $w=\omega_1\dots\omega_k\in\W$ let us set $F_{w}:(\N\setminus\{0\})^{k}\rightarrow \R$ the function defined by:
			\[
			F_w(n_1,\dots, n_k)\coloneqq \prod_{i=1}^{k} f_{w_i}(n_i).
			\]
			We call a \emph{formal series} a pair $(A \otimes w ,F)$ of the set:
			\begin{equation}\label{eq:formal_series_building}
			\left\langle \faktor{\R\mathcal{P}(\N^{\ell(w)})}{J_{\ell(w)}}\otimes w  \,\middle|\, w\in\W \right\rangle \times (\N^* \rightarrow \R)^{\N^*},
			\end{equation}
		where  $\calP(X)$ is the power set of $X$ and for all $n\in\N,~J_n$ is the ideal of the algebra $\left(\K\mathcal{P}(\N^n), +, \cup, 0 ,\ew{}\right)$ generated by:
		\[
		\left\{A\cup B -A-B+A\cap B \,\middle|\, A,B\in \mathcal{P}(\N^n)\right\}.
		\]
		A formal series $(A\otimes w,F)$ with $w=w_1\dots w_k$ will be written:
			\[
			\sum_A \left(\prod_{i=1}^k f_{w_i}\right).
			\] 
			The omission of the summation variable is here to distinguish formal series from usual ones. 
		We denote by $\FS_{\N}$ the set of formal series over $\N$.
	\end{defi}			
	
	\begin{Rq}
		The notation of formal series is made such that a specification of $F$ gives a concrete series to compute. Note that we are not evaluating it, we only consider it as a formal object.
	\end{Rq}
	\begin{Not}\label{Not:cart}
		For sake of simplicity, a formal series will be denoted $A\otimes w\times F$ instead of $(A\otimes w,F)$.
	\end{Not}
	The aim of this construction is to consider a formal object previous to any evaluation of \MZV{}.
\begin{Eg}
		In this section, we will mainly apply this construction to $F:k\rightarrow (n \rightarrow n^{-k})$. Hence, a formal version of the following \MZVs{} 
		\[
		\sum_{1 \leq n_2 < n_1} \frac{1}{n_1^2}\frac{1}{n_2}
		\]
		 is given by:
		 \[
		 \{1\leq n_2 < n_1\} \otimes \mathbf{2}~\mathbf{1}\times F.
		 \]
		 This is why we use the series notation.
	\end{Eg}
    Our formal series have some of the usual properties of series.
	\begin{Prop}\label{prop:formal_series_properties}
		Let $d\in\N^*$ and $F\in(\N^* \rightarrow \R)^{\N^*}$.
		The following properties of series hold for our formal series:
		\begin{itemize}
			\item let $A,B\in\mathcal{P}(\N^d)$ and $w=\omega_1\dots \omega_d\in \W$:
			\[
			\sum_{A\cup B} \left(\prod_{i=1}^d f_{\omega_i} \right)=\sum_A \left( \prod_{i=1}^d f_{\omega_i} \right)+\sum_B \left( \prod_{i=1}^d f_{\omega_i} \right)-\sum_{A\cap B}  \left(\prod_{i=1}^d f_{\omega_i}\right).
			\]
			\item let $w_1=\omega_1^{(1)}\dots \omega_1^{(d)},~w_2=\omega_2^{(1)}\dots \omega_2^{(d)}\in \W$ and $A\in\mathcal{P}(\N^d)$, then:
			\begin{align*}
				&\sum_A \left( \prod_{i=1}^d f_{\omega_1^{(i)}}+\prod_{j=1}^n f_{\omega_2^{(j)}} \right) 
				=\sum_A \left( \prod_{i=1}^d f_{\omega_1^{(i)}}\right)
				+\sum_A \left( \prod_{j=1}^d f_{\omega_2^{(j)}}\right).
			\end{align*}
		\end{itemize}
	\end{Prop}
	\begin{proof}
		Let $d\in\N^*,w=\omega_1\dots \omega_d\in\W$ and $A,B\in\mathcal{P}(\N^d)$, hence by definition~\ref{defi:formal_series}:
		\begin{align*}
			\sum_{A\cup B} \left(\prod_{i=1}^d f_{\omega_i} \right)&=(A\cup B) \otimes w \times F \\
			&= A\otimes w \times F + B \otimes w \otimes F -A\cap B\otimes w\times F\\
			&= \sum_A \left( \prod_{i=1}^d f_{\omega_i} \right)+\sum_B \left( \prod_{i=1}^d f_{\omega_i} \right)-\sum_{A\cap B}  \left(\prod_{i=1}^d f_{\omega_i}\right).
		\end{align*}
		Finally, let us consider two words $w_1=\omega_1^{(1)}\dots \omega_1^{(d)},~w_2=\omega_2^{(1)}\dots \omega_2^{(d)}$ and $A\in\mathcal{P}(\N^d)$. Hence, by definition~\ref{defi:formal_series}:
		
		\begin{align*}
			\sum_A \left( \prod_{i=1}^d f_{\omega_1^{(i)}}+\prod_{j=1}^n f_{\omega_2^{(j)}} \right) &= A \otimes (w_1 + w_2) \times F \\
			&=A \otimes w_1 \times F + A \otimes w_2 \times F \\
			&=\sum_A \left( \prod_{i=1}^d f_{\omega_1^{(i)}}\right)
			+\sum_A \left( \prod_{j=1}^d f_{\omega_2^{(j)}}\right).
			\end{align*}
		So, it fulfils those usual properties of the series.
	\end{proof}
	
	\subsection{The tridendriform algebra of formal series}
	
	\subsubsection{Formal series for \MZVs{}}
	
	In this section, we focus on the map: 
	\begin{align} \label{eq:Our_F}
	 F:\left\lbrace\begin{array}{rcl}
	 	\N^* & \longrightarrow (\N^*\rightarrow\R), \\
	 	k & \longmapsto (n\mapsto n^{-k}).
	 \end{array} \right.
	\end{align}
		 We will omit the dependency on $F$ in this section to lighten the notation. Note that some of the results are still true with other appropriate $F$.
	
	\begin{defi} \label{def:formal_series_MZVs}
		We define $\Suum_{\N^*}$ the subspace of $\FS(\N^*)$ generated by:
		\[
		\bigcup_{r\geq 0}\bigcup_{\omega_1\dots \omega_r\in \W} \left\{ \sum_{1\leq n_r< \dots < n_1} \left(\prod_{i=1}^r f_{\omega_i}\right) \right\}.
		\]
		\end{defi}
		Since $F$ is fixed  $\Suum_{\N^*}$ inherits the vector space structure from the left part of equation~\eqref{eq:formal_series_building}.
		Hence, $\Suum_{\N^*}$ is the set of formal series needed to build \MZVs{} in their series version (equation~\eqref{def:MZVs}).
	
	It is well-known, from standard analysis tools, that a formal series $A\otimes\omega\times F\in\Suum_{\N^*}$ is actually equivalent to a convergent series provided the first letter of $\omega$ is not one. Then the following map will encode all the analysis job to get any real value.
	\begin{defi}[Evaluation map]\label{defi:ev_map_tri}
		 A formal series $A\otimes\omega\times F\in\Suum_{\N^*}$ is said to be \emph{convergent} if $\omega=\emptyset$ or if its first letter is not one. We write $\Suumcv$ these formal series. Then we define the \emph{evaluation map}, denoted $\ev$, as:
		\begin{equation*}
				\ev:\left\lbrace\begin{array}{rcl}
				\Suumcv & \rightarrow & \R, \\
				A \otimes \omega_1\dots \omega_k \times F & \mapsto & \displaystyle \sum_{(n_1,\dots,n_k)\in A}\left( \displaystyle\prod_{i=1}^k f_{\omega_i}(n_i) \right),
			\end{array} \right.
		\end{equation*}
		where $F$ is the function of equation~\eqref{eq:Our_F}.
	\end{defi}
	\begin{Rq}\label{Rq:not_linear}
	Notice that ev could a priori be defined for other domains than the ones of definition~\ref{def:formal_series_MZVs}. This extended map is a linear map for the first variable and is homogenous on the second variable. For instance, take for example $A=\{(2,2)\}$, $w=\mathbf{2}~\mathbf{2}$ with ${F(k)=(n\mapsto n^{-k})}$. Then $\ev(A\otimes w\times F)=\frac{1}{2^2}.\frac{1}{2^2}=\frac{1}{16}$. But it is not linear for the second argument, for instance let $\lambda\in\R, \lambda F(k)=(n\mapsto \lambda n^{-k})$, thus \[
	\ev(A\otimes w\times \lambda F)=\lambda\frac{1}{2^2}.\lambda \frac{1}{2^2}=\frac{1}{16}\lambda^2.
	\]
	More precisely, for $A\otimes w$ fixed it satisfies for any $\lambda\in\R$:
	\[
	\ev(A\otimes w \times \lambda F)=\lambda^{|w|}\ev(A\otimes w \times F).
	\]
	\end{Rq}

	\begin{Eg}\label{Eg:ev_tridend}
			More formally, to evaluate such a formal integral, one needs to get rid of the indeterminate forms. For instance, for $F$ such that for any $k\in\N^*, f_k:n\mapsto n^{-k}$, we have:
			\begin{align*}
				\frac{\pi^2}{6}=& \ev\left(\sum_{n\geq 1} \frac{1}{n^2}\right)
				=\ev(\{n\geq 1\} \otimes \mathbf{2}\times F) \\
				=&\ev\left(\sum_{n\geq 1} \frac{1-n}{n^2} +\sum_{n\geq 1} \frac{1}{n}\right)
				=\ev\left(\sum_{n\geq 1} \frac{1}{n^2}-\sum_{n\geq 1} \frac{1}{n} +\sum_{n\geq 1} \frac{1}{n}\right) \\
				=&\ev(\{n\geq 1\} \otimes \mathbf{2}\times F-\{n\geq 1\} \otimes \mathbf{1}\times F+\{n\geq 1\} \otimes \mathbf{1}\times F).
			\end{align*}
			Note that as $\sum_{n\geq 1} \frac{1-n}{n^2}$ and $\sum_{n\geq 1} \frac{1}{n}$ are not convergent formal series one can not apply the linearity property of $\ev$. Hence, it prevents us from writing $\infty - \infty$.
	\end{Eg}
	
\subsubsection{The tridendriform structure}
	
	Using an analogous of Chen's iterated integrals lemma~\cite{Chen1977} but for sums, we define the unique product $\qshuffle$ over $\Suum_{\N^*}$ such that $\ev$ is a morphism for the multiplicative structure of $\R$:
	\begin{defi}
		We define three new linear operators $\prec,\succ$ and $\cdot$ defined for any $\alpha,\beta\in\Suum_{\N}$ such that:
		\begin{align*}
			\alpha=\sum_{1\leq n_k< \dots < n_1} \prod_{i=1}^{k} f_{\omega_i}, &&
			\beta=\sum_{1\leq n_{k+l}< \dots < n_{k+1}} \prod_{i=k+1}^{k+l} f_{\omega_i},
		\end{align*}
		by:
		\allowdisplaybreaks[3]
		\begin{align*}
			\alpha\prec\beta&\coloneqq\sum_{\substack{\sigma\in\QSh(k,l) \\ \sigma^{-1}(\{1\})=\{1\}}}\{1\leq n_{\max(\sigma)}< \dots < n_{1} \}\otimes \omega_{\sigma^{-1}(\{1\})}\dots \omega_{\sigma^{-1}(\{\max(\sigma)\})} \times F, \\
			&=\sum_{\substack{\sigma\in\QSh(k,l) \\ \sigma^{-1}(\{1\})=\{1\}}}\sum_{1\leq n_{\max(\sigma)}< \dots< n_{1}} \prod_{i=1}^{\max(\sigma)} f_{\omega_{\sigma^{-1}(\{i\})}}, \\
			\alpha\succ \beta&\coloneqq\sum_{\substack{\sigma\in\QSh(k,l) \\ \sigma^{-1}(\{1\})=\{k+1\}}}\{1 \leq n_{\max(\sigma)}< \dots < n_{1} \}\otimes \omega_{\sigma^{-1}(\{1\})}\dots \omega_{\sigma^{-1}(\{\max(\sigma)\})} \times F, \\
			&=\sum_{\substack{\sigma\in\QSh(k,l) \\ \sigma^{-1}(\{1\})=\{k+1\}}}\sum_{1\leq n_{\max(\sigma)}< \dots< n_{1}} \prod_{i=1}^{\max(\sigma)} f_{\omega_{\sigma^{-1}(\{i\})}}, \\ \\
			\alpha\cdot\beta&\coloneqq\sum_{\substack{\sigma\in\QSh(k,l) \\ \sigma^{-1}(\{1\})=\{1,k+1\}}}\{1\leq n_{\max(\sigma)}< \dots < n_{1} \}\otimes \omega_{\sigma^{-1}(\{1\})}\dots \omega_{\sigma^{-1}(\{\max(\sigma)\})}\times F, \\
			&=\sum_{\substack{\sigma\in\QSh(k,l) \\ \sigma^{-1}(\{1\})=\{1,k+1\}}}\sum_{1\leq n_{\max(\sigma)}< \dots< n_{1}} \prod_{i=1}^{\max(\sigma)} f_{\omega_{\sigma^{-1}(\{i\})}}, \\
		\end{align*}
		\allowdisplaybreaks[0]
		where for any $(i,j)\in\IEM{1}{k}\times\IEM{k+1}{k+l}, \omega_{\{i,j\}}=\omega_i+\omega_j$.
		Moreover, we denote:
		\begin{align*}
			\alpha\qshuffle \beta &\coloneqq\alpha\prec\beta+\alpha\succ \beta+\alpha\cdot\beta \\
			&=\sum_{\sigma\in\QSh(k,l)} \sum_{1\leq n_{\max(\sigma)}< \dots\leq n_{1}} \prod_{i=1}^{\max(\sigma)} f_{\omega_{\sigma^{-1}(\{i\})}}.
		\end{align*}
	\end{defi}
		\begin{Prop}\label{Prop:tridend_Suum}
			The structure $(\Suum_{\N^*},\prec,\succ,\cdot)$ is a tridendriform algebra.
		\end{Prop}
	\begin{proof}
		Let $\alpha,\beta$ and $\gamma$ be three elements of $\Suum_{\N^*}$:
		\begin{align*}
			\alpha=\sum_{1\leq n_k< \dots < n_1} \prod_{i=1}^{k} f_{\omega_i}, &&
			\beta=\sum_{1\leq n_{k+l}< \dots < n_{k+1}} \prod_{i=k+1}^{k+l} f_{\omega_i}, &&
			\gamma=\sum_{1\leq n_{k+l+s}< \dots < n_{k+l+1}} \prod_{i=k+1}^{k+l} f_{\omega_i}.
		\end{align*}
		We need to show those seven relations:
		\begin{gather*}
			(\alpha\prec \beta)\prec \gamma =\alpha\prec(\beta\qshuffle\gamma),\\
			(\alpha\succ \beta)\prec \gamma =\alpha\succ(\beta\prec \gamma),\\
			(\alpha\qshuffle \beta)\succ \gamma =\alpha\succ(\beta\succ \gamma),\\
			(\alpha\succ \beta)\cdot \gamma =\alpha\succ(\beta\cdot \gamma), \\
			(\alpha\prec \beta)\cdot \gamma =\alpha\cdot(\beta\succ \gamma), \\
			(\alpha\cdot \beta)\prec \gamma =\alpha\cdot(\beta\prec \gamma), \\
			(\alpha\cdot \beta)\cdot \gamma =\alpha\cdot (\beta\cdot \gamma).
		\end{gather*}
		Note that applying a quasi-shuffle to those quantities is just equivalent to quasi-shuffle the words $\omega_1\dots \omega_k$,  $\omega_{k+1}\dots \omega_{k+l}$ and $\omega_{k+l+1}\dots \omega_{k+l+s}$. Actually, considering the map $F$ of equation~$\eqref{eq:Our_F}$ there is the following bijection:
		\[
		\Phi_F:\left\lbrace \begin{array}{rcl}
			\W & \rightarrow & \Suum_{\N^*}, \\
			\omega_1\dots \omega_k & \mapsto & \left\{ 1\leq n_k < \dots < n_1 \right\}\otimes \omega_1\dots \omega_k \times F
		\end{array}  \right.
		\]
		extended by linearity. 
		This map is clearly a morphism for the products $\prec$, $\succ$ and $\cdot$ in words and the same products in $\Suum_{\N}$. Thus the 
        relations hold thanks to this bijection.
	\end{proof}
	\begin{Prop}\label{lem:gen_tridend}
		 Then, $(\Suum_{\N^*},\prec,\succ,\cdot)$ is generated as a tridendriform algebra by the following set:
		\[
		\left\{  \sum_{{1\leq n}} f_{t}\right\}_{t\in\N}.
		\]
	\end{Prop}
	\begin{proof}
		This is a consequence of the fact that the set of words is generated by its letters as a tridendriform algebra and the existence of the isomorphism of tridendriform algebras $\Phi_F$ build above.
	\end{proof}
	Finally, the following well-known result states that the evaluation map of definition~\ref{defi:ev_map_tri}
	respects the associative algebra structure:
	\begin{Prop} \label{prop:ev_alg_morph}
	 The space $\Suumcv$ is a tridendriform subalgebra of $\Suum_{\N^*}$, and $\ev:\Suumcv\longrightarrow\R$ is an algebra morphism for the quasi-shuffle product $\qshuffle$.
	\end{Prop}
	\begin{proof}
	 If $w$ and $w'$ are two words that are either empty or do not start with $1$, then clearly so are $w\succ w',w\cdot w'$ and $w\prec w'$. Thus $\Suumcv$ is a tridendriform subalgebra. The fact that ev is an algebra morphism for the quasi-shuffle product is a well-known fact of the theory of \MZVs{} that can be proved using standard analytical techniques. It is essentially a reformulation of the first equality in equation~\eqref{eq:shuffle_stuffle_zeta}. 
	\end{proof}

	\subsection{Tridendriform structure and \MZVs{}}
	
    In this subsection, $F$ will still be the map of equation~\eqref{eq:Our_F}.
	\begin{defi}[\MZV{} series representation] \label{def:zeta_FS}
		We define the linear map ${\zeta_{\FS}:\Wcv \rightarrow \Suum_{\N^*}}$ defined for any word $w=\omega_1\dots \omega_k\in \W$ by:
		\[
		\zeta_{\FS}(w)=\sum_{0\leq n_k< \dots < n_1}\left( \prod_{i=1}^k f_{\omega_i}\right)=A\otimes w\times F
		\]
		with $\zeta_{\FS}(\emptyset)=\emptyset\otimes\emptyset \times F$ and with $A$ given by the domain in the definition of $\Suum_{\N^*}$
		(definition~\ref{def:formal_series_MZVs}).
		Hence, the classical \emph{\MZV{}} of definition~\ref{def:MZVs} satisfies:
		\[
		\zeta(w)=\ev\circ \zeta_{\FS}.
		\]
	\end{defi}

	Our motivation for introducing formal series as an intermediate step for building \MZVs{} comes from the following simple yet important proposition.
	\begin{Prop}\label{Lem:zeta_morph_tridend}
		The map $\zeta_{\FS}:\W\rightarrow \Suum_{\N}$ is a tridendriform morphism.
	\end{Prop}
	\begin{proof}
		Using the fact that $\prec=\succ^{\textrm{op}}$, it is enough to prove that $\zeta_{\FS}$ is a morphism for $\prec$ and $\cdot$. Let $v,u\in (\N\setminus\{0\})^{\star}$ such that $v=\omega_1\dots \omega_k$ and $u=\omega_{k+1}\dots \omega_{k+l}$. Hence:
		\begin{align*}
			\zeta_{\FS{}}(v\prec u)&=\sum_{\substack{\sigma\in\QSh(k,l) \\ \sigma^{-1}(\{1\})=\{1\}}} \zeta_{\FS}\left( \omega_{\sigma^{-1}(\{1\})}\dots \omega_{\sigma^{-1}(\{\max(\sigma))\}} \right) \\
			&=\sum_{\substack{\sigma\in\QSh(k,l) \\ \sigma^{-1}(\{1\})=\{1\}}} \sum_{1\leq n_{\max(\sigma)}< \dots < n_{1}} \prod_{i=1}^k f_{\omega_{\sigma^{-1}(\{i\})}} \\
			&=\left( \sum_{1\leq n_k< \dots < n_1} \prod_{i=1}^k f_{\omega_i} \right)\prec \left( \sum_{1\leq n_{k+1}< \dots < n_{k+l}} \prod_{j=k+1}^{k+l} f_{\omega_j} \right) \\
			&=\zeta_{\FS}(v)\prec \zeta_{\FS}(u).
		\end{align*}
		The third equality comes from the definition of the product $\prec$ in $\Suum_{\N^*}$.
		
		We need to check that $\zeta_{\FS}$ is also a morphism for $\cdot$. With the same notations, we compute:
		\begin{align*}
			\zeta_{\FS{}}(v\cdot u)&=\sum_{\substack{\sigma\in\QSh(k,l) \\ \sigma^{-1}(\{1\})=\{1,k+1\}}} \zeta_{\FI}\left( \omega_{\sigma^{-1}(\{1\})}\dots \omega_{\sigma^{-1}(\{\max(\sigma)\})} \right) \\
			&=\sum_{\substack{\sigma\in\QSh(k,l) \\ \sigma^{-1}(\{1\})=\{1,k+1\}}} \sum_{1\leq n_{\max(\sigma)}< \dots < n_{1}} \prod_{i=1}^{\max(\sigma)} f_{\omega_{\sigma^{-1}(\{i\})}} \\
			&=\left( \sum_{1\leq n_k< \dots < n_1} \prod_{i=1}^k f_{\omega_i} \right)\cdot \left( \sum_{1\leq n_{k+l}< \dots < n_{k+1}} \prod_{j=k+1}^{k+l} f_{\omega_j} \right) \\
			&=\zeta_{\FS}(v)\cdot \zeta_{\FS}(u).
		\end{align*} 
		This achieves the proof.
	\end{proof}
	This proposition indicates that the well-known property that \MZVs{} form an algebra morphism for the quasi-shuffle product comes from a deeper tridendriform structure. This structure is lost when one uses the evaluation map to obtain real numbers since $\R$ has no tridendriform structure.
	
	In the following, using the universal properties of the free tridendriform algebra over \Sch{} trees and binary trees, we extend the definition of the formal \MZVs{} we have just introduced.

\section{\MZVs{} and Tridendriform structures} \label{sec:three}

\subsection{Schroeder trees}

	Let us quickly introduce the combinatorial view of the free tridendriform algebra. Its definition over one generator can be found in the first author's article~\cite{Catoire_23}. In order to describe the free tridendriform algebra with several generators, we introduce Schroeder trees.
	\begin{defi}\label{defi:schroeder}
	Recall that a \emph{planar rooted tree} is a planar connected directed acyclic graph such that there is at most one path between any two vertices, and such that there is a unique minimal element for the relation on its set of vertices
	\begin{equation*}
 	v\leq v':\Longleftrightarrow \text{there is path from $v$ to $v'$}.
	\end{equation*}

	For any tree $t$, a \emph{leaf} of $t$ is a vertex that is maximal for the poset relation $\leq$. We say one of its vertex is \emph{internal} if it is not a leaf.
	
	So for any tree, the minimal element is called the \emph{root} of the tree and its edges are oriented from the root to the leave.
	
	A \emph{Schroeder} tree is a planar rooted tree such that any of its internal vertices has at least two children. We write $\Schtree$ the set of Schroeder trees.
	
	Let $t\in \Schtree$ a Schroeder tree. Let $\nu(t)$ be
	 the set of internal vertices of~$t$. An \emph{angle of $v$} is a pair $(e,e')$ of  consecutive (for the reading order from left to right) edges starting from $v$.
	For any $v\in\nu(t)$, we define $\angles(v)$ the set of angles of $v$ in $t$.
	
	We will denote $\angles(t)\coloneqq\displaystyle\bigcup_{v\in\nu(t)} \angles(v)$.
	
	Let $t$ be a Schroeder tree and $\Omega$ a set. An \emph{angle decoration map} for $t$ is a map ${D:\angles(t)\rightarrow \Omega}$.
A \emph{$\Omega$-decorated Schroeder tree} is a pair $(t,D)$ with $t$ a Schroeder tree and $D$ an angle decoration map for $t$. We write $\Schtree(\Omega)$ the set of $\Omega$-decorated Schroeder trees.
	\end{defi}
\begin{Rq}
 In general, we omit to write down the decoration map of a decorated Schroeder tree, and instead write the decorations on the angles. Alternatively, if $|\angles(v)|=n$, we can write $v$ as being decorated by the word $\omega_1\cdots\omega_n$ instead of $\omega_i$ being the $i$th decoration of the $i$th angle of $v$ (read from left to right).
\end{Rq}

	We give below some examples of decorated and undecorated Schroeder trees and refer the reader to~\cite{zhang2020free} for a detailed presentation.
\begin{Eg}
	Here are some examples of such trees where $\Schtree_n$ is the set of Schroeder trees with exactly $n+1$ leaves.
		\begin{align*}
			\Schtree_0=\left\lbrace |\right\rbrace, && \Schtree_1=\left\lbrace\labY{}\right\rbrace, && \Schtree_2=\left\lbrace 
			\labbalais{}{},
			\labbalaisg{}{} ,
			\labbalaisd{}{}
			\right\rbrace.
		\end{align*}
        Setting $T_n(\Omega)$ the set of $\Omega$-decorated Schroeder trees with exactly $n+1$ leaves we have
        \begin{align*}
			\Schtree_0(\Omega)=\left\lbrace |\right\rbrace, && \Schtree_1(\Omega)=\left\lbrace\labY{$\omega$\,}\middle|\,\omega\in\Omega\right\rbrace, && \Schtree_2(\Omega)=\left\lbrace 
			\labbalais{$\omega$}{$\alpha$},
			\labbalaisg{$\omega$}{$\alpha$} ,
			\labbalaisd{$\omega$}{$\alpha$}\,\middle|\,(\omega,\alpha)\in\Omega^2
			\right\rbrace.
		\end{align*}
\end{Eg}

\subsection{Free tridendriform algebras}

The free tridendriform algebra over $\N^*$ has for underlined vector space the space $\Schtree(\N^*)$.

	Let $t,s$ be two trees different from $|$. We see $t$ as a right comb and $s$ as a left comb. In other words, we put:
	\begin{center}
		\begin{align*}
			t&=\peignedroitdec{F_1}{F_2}{F_k}{w_1}{w_2}{w_k}
			& \text{ and } &&  s=\peignegauchedec{F_{k+1}}{F_{k+2}}{F_{k+l}}{w_{k+1}}{w_{k+2}}{w_{k+l}},
		\end{align*}
	\end{center}
	where for all $i\in\IEM{1}{k+l}, F_i$ is a forest with decorated angles and the $w_i$ is the word of integer labels decorating the angles between the elements of the forest $F_i$. Here, $k$ represents the number of nodes on the rightmost branch of $t$ and $l$ is the number of nodes on the leftmost branch of $s$.
	
	\begin{defi}\label{def:quasiaction}
		Let $t$ and $s$ be two decorated trees whose respective right comb and left comb representations are given as above. Put $k$ (respectively $l$) the number of nodes on the rightmost branch (respectively the leftmost)  of $t$ (respectively of $s$). Let $\sigma\in\QSh(k,l)$ whose image is $\IEM{1}{n}$. We define $\sigma(t,s)$ the tree obtained the following way:

		\begin{enumerate}
			\item  let us start from the tree:
			
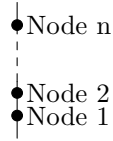
\begin{center}
				\begin{tikzpicture}[line cap=round,line join=round,>=triangle 45,x=0.3cm,y=0.3cm]
					\begin{scope}{shift={(-3,0)}}
						\draw (0,0)--(0,2.5);
						\draw[dashed] (0,2.5)--(0,5);
						\draw (0,5)--(0,6);
						\filldraw [black] (0,1) circle (2pt) node[anchor=west]{Node 1};
						\filldraw [black] (0,2) circle (2pt) node[anchor=west]{Node 2};
						\filldraw [black] (0,5) circle (2pt) node[anchor=west]{Node n};
					\end{scope}
				\end{tikzpicture}
			\end{center}
			\item For $i\in\IEM{1}{k},$ we graft $F_i$ as the \emph{left} son of the node $\sigma(i)$;
			\item For $i\in\IEM{k+1}{k+l},$ we graft $F_i$ as a \emph{right} son of the node $\sigma(i)$;
			
			\item For the decorations, we proceed the following way :
			\begin{itemize}
				\item 	all the forests $F_i$ keep their original decorations;
				\item 	Let $m\in\IEM{1}{n}$. The node $m$ of the ladder is decorated by $w_i$ if $i$  is the unique preimage of $m$ by $\sigma$. Otherwise, there exists a unique  $(i,j)\in\IEM{1}{k}\times\IEM{k+1}{k+l}$ such that $m=\sigma(i)=\sigma(j)$, then we decorate  the node $m$ of the ladder with $w_i w_j.$
			\end{itemize}
		\end{enumerate}
	\end{defi}
	\begin{Eg}\label{Eg:quasi_action}
		Consider $\sigma=(1,3,2,3)$ a $(2,2)$-quasi-shuffle.
		Take
		$t=$	\raisebox{-0.3\height}{\begin{tikzpicture}[line cap=round,line join=round,>=triangle 45,x=0.3cm,y=0.3cm]
				\draw (0,0)--(0,1);
				\draw (0,1)--(-1,2) node[left]{$F_1$};
				\draw (0,1)--(2,3);
				\draw (1,2)--(0,3) node[left,above]{$F_2$};
				\draw (0,1) node[right]{\scriptsize{$w_1$}};
				\draw (1,2) node[right]{\scriptsize{$w_2$}};
		\end{tikzpicture}} and ${s=\raisebox{-0.3\height}{\begin{tikzpicture}[line cap=round,line join=round,>=triangle 45,x=0.3cm,y=0.3cm]
					\draw (0,0)--(0,1);
					\draw (0,1)--(1,2) node[right]{$F_{3}$};
					\draw (0,1)--(-2,3);
					\draw (-1,2)--(0,3) node[right,above]{$F_{4}$};
					\draw (0,1) node[left]{\scriptsize{$w_3$}};
					\draw (-1,2) node[left]{\scriptsize{$w_4$}};
			\end{tikzpicture}}}$.
		Then:
		\begin{center}
			$\sigma(t,s)=$\raisebox{-0.5\height}{\begin{tikzpicture}[line cap=round,line join=round,>=triangle 45,x=0.5cm,y=0.5cm]
					\draw (0,0)--(0,4);
					\draw (0,1)--(-1.5,2) node[left]{$F_1$};
					\draw (0,2)--(1,3) node[right]{$F_3$};
					\draw (0,3)--(-1,4) node[left]{$F_2$};
					\draw (0,3)--(1,4) node[right]{$F_4$};
					\draw (0,1) node[right]{\scriptsize{$w_1$}};
					\draw (0,2) node[left]{\scriptsize{$w_3$}};
					\draw (0,3) node[below,left]{\scriptsize{$w_2 w_4$}};
			\end{tikzpicture} }.
		\end{center}
	\end{Eg}
	An explicit, non-inductive description of the tridendriform structure of Schroeder trees was given in~\cite{Catoire_23}.
\begin{thm}[\cite{Catoire_23}]\label{thm:produit}
	Let $t,s$ be two trees different from $|$ as described above. Then for any set $\Omega$, $\Schtree(\Omega)$ has the structure of a tridendriform algebra given by: 
	\[
	t*s=\sum_{\sigma\in \QSh(k,l)}\sigma(t,s).
	\]
	Moreover:
	\begin{align*}
		t\prec s=\sum_{\substack{\sigma\in\QSh(k,l) \\ \sigma^{-1}(\{1\})=\{1\}}} \sigma(t,s), &&
		t\cdot s=\sum_{\substack{\sigma\in\QSh(k,l) \\ \sigma^{-1}(\{1\})=\{1,k+1\}}} \sigma(t,s), &&
		t\succ s=\sum_{\substack{\sigma\in\QSh(k,l) \\ \sigma^{-1}(\{1\})=\{k+1\}}} \sigma(t,s).
	\end{align*}
\end{thm}
The three products $\prec,\succ $ and $\cdot$ also admit an inductive description that we now present.
	\begin{Not}
		Let $k\in \N, k\geq 2$ and $t_1,\dots, t_k$ a family of $k$ trees. We denote by ${t_1\vee_{\omega_1}  \dots \vee_{\omega_{k-1}} t_k}$ the tree obtained grafting in this order $t_1, \dots, t_k$ to a common root decorated by $\omega_1\dots \omega_{k-1}.$ 
	\end{Not}
	\begin{Eg}
		With this notation, all trees can be represented only with Y-shaped trees and the tree |:
		\[
		\labY{1} \vee_2 | = \labbalaisg{2}{1}.
		\]
	\end{Eg}
	\begin{Lemme}\label{Lem:induction_product}
		Let $t$ and $s$ be two \Sch{} trees different from $|$ with:
		\begin{align*}
			t=t^{(1)}\vee_{\omega_1} \dots \vee_{\omega_{k-1}} t^{(k)}, &&
			s=s^{(1)} \vee_{\mu_1} \dots \vee_{\mu_{l-1}} s^{(l)}.
		\end{align*}
		Then:
		\begin{gather*}
			t\prec s=t^{(1)}\vee_{\omega_1} \dots \vee_{\omega_{k-1}} \left(t^{(k)}*s\right), 
			t \succ s= \left(t * s^{(1)}\right)\vee_{\mu_1} s^{(2)}\vee_{\mu_2} \dots \vee_{\mu_{l-1}} s^{(l)},\\
			t\cdot s= t^{(1)}\vee_{\omega_1} \dots \vee_{\omega_{k-1}} \left(t^{(k)}* s^{(1)}\right)\vee_{\mu_1} s^{(2)}\vee_{\mu_2} \dots \vee_{\mu_{l-1}} s^{(l)}.
		\end{gather*}
	\end{Lemme}
	This lemma allows us to do proofs by induction over the number of leaves of a tree.
	
	\subsection{Tridendriform zeta values}
	
	The tridendriform algebra of theorem~\ref{thm:produit} is the free tridendriform algebra over $\Omega$.
	\begin{thm}[\cite{zhang2020free}] \label{thm:univ_prop} 
	Let $\Omega$ be a set. $(\Schtree(\Omega),\prec,\succ,.)$ is the free tridendriform algebra over $\Omega$: for any tridendriform algebra $A$ and any map $f:\Omega\longrightarrow A$, there exists a unique morphism of tridendriform algebras $\Phi:\Schtree(\Omega)\longrightarrow A$ such that the diagram below commutes, with ${i:\Omega\longrightarrow \Schtree(\Omega)}$ the canonical embedding defined by $i(n)\coloneqq\labY{n}$
\begin{figure}[ht]
	\centering 
	\begin{tikzcd}
	\Omega \arrow[rd, "f"] \arrow[r, "i"] & \Schtree(\Omega) \arrow[d, "\Phi"] \\
	& A
	\end{tikzcd}.
\end{figure}
\end{thm}

\begin{Not}
 We will write $\Tridend(\Omega)$ the free tridendriform algebra over $\Omega$ described above.
\end{Not}
	Applying the freeness property of $\Tridend(\N^*)$ to the tridendriform algebra $\calS_{\N^*}$ (definition\ref{def:formal_series_MZVs}), we obtain tridendriform zeta values.
	\begin{defi} \label{def:triden_zeta}
	 Let $\sum:\N^*\longrightarrow\calS_{\N^*}$ be the map defined by $\sum(n)\coloneqq\sum_{k\geq1} f_n$, where ${f_n(m)\coloneqq m^{-n}}$ as in definition~\ref{def:formal_series_MZVs}. Then the \emph{tridendriform zeta values} map is the unique morphism of tridendriform algebras $\zeta^{\rm Tri}:\Tridend(\N^*)\longrightarrow\calS_{\N^*}$ whose existence and unicity is given by the universal property of $\Tridend(\N^*)$ stated in theorem~\ref{thm:produit}.
	\end{defi}
	Note that the map that we built is by definition a morphism of tridendriform algebras. Therefore, it is also a morphism for the quasi-shuffle product $*$ of Schroeder trees which is explicitly described in theorem~\ref{thm:produit}:
	\begin{equation} \label{eq:tridend_zeta_morphism}
	 \forall(t,s)\in\Tridend(\N^*),~\zeta^{\rm Tri}(t*s)=\zeta^{\rm Tri}(t)*\zeta^{\rm Tri}(s).
	\end{equation}
	Composing asuitably restricted $\zeta^{\rm Tri}$ with the evaluation map of definition~\ref{defi:ev_map_tri} gives thanks to proposition~\ref{prop:ev_alg_morph} a map with values in $\R$ which is an algebra morphism. The aforementioned restriction needs to be taken to avoid divergent series, and this will be expanded upon later on.
	In other words, we have build a new generalisation of \MZVs{} which is an algebra morphism for a generalisation of the quasi-shuffle product on words.
	
	In the rest of this section, we will aim at relating $\zeta^{\rm Tri}$ to other known objects, and in particular Arborified Zeta Values and \MZVs{}. This will also solve the issue of convergence of Tridendriform Zeta Values, which we have left pending so far.

	\section{From tridendriform to Zeta Values}\label{sec:tridend_to_usual}
	
\subsection{From the free algebra to usual trees}

As we have seen the free vector space of Schroeder trees with angles decorated by non-negative integers has a tridendriform structure explained above making it the \emph{free tridendriform algebra} over $\N^*$ denoted $\Tridend(\N^*)$. Our aim is to recover classical trees for the \MZVs{} theory.

	Now, we introduce the main objects that we will care about:
	\begin{defi}
		We define $\Tree(\N^*)$ the vector space generated by Schroeder trees whose internal vertices are decorated by an element of $\N^*$ \emph{greater than the number of {angles of this vertex}}.

		Hence, one can see an element of $\Tree(\N^*)$ as the data $(t,d)$ of a \Sch{} tree $t$ and a decoration map $d:\nu(t)\rightarrow \N^*.$
	\end{defi}
	\begin{Eg}
		For instance:
		$\labY{1}, \labY{3},  \labbalaisd{1}{1}, \labbalaisbis{2} \text{ and }\labbalaisg{2}{1}$ are elements of $\Tree(\N^*)$ whereas $\labbalaisbis{1}$ is not.
	\end{Eg}	
	\begin{defi} \label{defi:iota}
		We define the map $\iota:\Tridend(\N^*) \rightarrow \Tree(\N^*)$ keeping the same tree structure but getting from the decoration of any angles a decoration of internal vertices by:
		\[
		\forall v\in\nu(t), d(v)=\sum_{a\in \angles(v)} D(a).
		\]
		In other words, $\iota(t,D)=(t,d)$.
	\end{defi}
	\begin{Eg}
		For instance:
		\[
		\iota\left(\labbalais{1}{2}\right)=\labbalaisbis{3}.
		\]
	\end{Eg}
	
	Consider $t$ and $s$ be two elements of $\Tree(\N^*)$:
	\begin{center}
		\begin{align}\label{eq:tree_shapes}
			t&=\peignedroitdec{F_1}{F_2}{F_k}{d_1}{d_2}{d_k}
			& \text{ and } &&  s=\peignegauchedec{F_{k+1}}{F_{k+2}}{F_{k+l}}{d_{k+1}}{d_{k+2}}{d_{k+l}},
		\end{align}
	\end{center}
	where for all $i\in\IEM{1}{k+l}, F_i$ is a decorated forest of elements of $\Tree(\N^*)$ and the $d_i$'s are  vertex decorations in $\N^*$. Then, $k$ represents the number of nodes on the rightmost branch of $t$ and $l$ is the number of nodes on the leftmost branch of $s$.
	
	Moreover, quasi-shuffles act on the set of pairs of elements of $\Tree(\N^*)$ by:
	\begin{defi}
		Let $t,s\in\Tree(\N^*)$ and $\sigma\in\QSh(k,l)$ where $k$ (respectively $l$) is the number of forests in the right (respectively left) comb representation of $t$ (respectively of $s$). Put ${n=\max(\sigma)}$. We define $\sigma(t,s)$ the element of $\Tree(\N^*)$ defined by:
		\begin{enumerate}
			\item the underlying tree is the same built in definition~\ref{def:quasiaction};
			\item the decoration function is unchanged for any forests expect for the vertices $v_i$ for $i\in\IEM{1}{n}$ of the ladder used to build the new tree in definition~\ref{def:quasiaction}:
			\[
			d(v_i)\coloneqq \begin{cases}
				d_t(v_j) &\text{if } \sigma^{-1}(\{i\})=\{j\} \text{ and } v_j\in\nu(t),\\
				d_s(v_j) &\text{if } \sigma^{-1}(\{i\})=\{j\} \text{ and } v_j\in\nu(s),\\
				d_t(v_{j_1})+d_s(v_{j_2}) &\text{if } \sigma^{-1}(\{i\})=\{j_1,j_2\} \text{ and } v_{j_1}\in\nu(t), v_{j_2}\in\nu(s),\\
			\end{cases}
			\]
			where $d_t$ and $d_s$ are respectively the decorations maps of $t$ and $s$.
		\end{enumerate}
	\end{defi}

	\begin{Eg}
		Let $\sigma=(1,3,2,3)$ be $(2,2)$-quasi-shuffle. 
		
		Take $t=$	\raisebox{-0.3\height}{\begin{tikzpicture}[line cap=round,line join=round,>=triangle 45,x=0.3cm,y=0.3cm]
				\draw (0,0)--(0,1);
				\draw (0,1)--(-1,2) node[left]{$F_1$};
				\draw (0,1)--(2,3);
				\draw (1,2)--(0,3) node[left,above]{$F_2$};
				\draw (0,1) node[right]{\scriptsize{$d_1$}};
				\draw (1,2) node[right]{\scriptsize{$d_2$}};
		\end{tikzpicture}} and $s=$ \raisebox{-0.3\height}{\begin{tikzpicture}[line cap=round,line join=round,>=triangle 45,x=0.3cm,y=0.3cm]
				\draw (0,0)--(0,1);
				\draw (0,1)--(1,2) node[right]{$F_{3}$};
				\draw (0,1)--(-2,3);
				\draw (-1,2)--(0,3) node[right,above]{$F_{4}$};
				\draw (0,1) node[left]{\scriptsize{$d_3$}};
				\draw (-1,2) node[left]{\scriptsize{$d_4$}};
		\end{tikzpicture}} two elements of $\Tree(\N^*)$.
		Then:
		\begin{center}
			$\sigma(t,s)=$\raisebox{-0.5\height}{\begin{tikzpicture}[line cap=round,line join=round,>=triangle 45,x=0.5cm,y=0.5cm]
					\draw (0,0)--(0,4);
					\draw (0,1)--(-1.5,2) node[left]{$F_1$};
					\draw (0,2)--(1,3) node[right]{$F_3$};
					\draw (0,3)--(-1,4) node[left]{$F_2$};
					\draw (0,3)--(1,4) node[right]{$F_4$};
					\draw (0,1) node[right]{\scriptsize{$d_1$}};
					\draw (0,2) node[left]{\scriptsize{$d_3$}};
					\draw (0,3) node[below,left]{\scriptsize{$d_2+ d_4$}};
			\end{tikzpicture} }.
		\end{center}
	\end{Eg}
	With this construction we get:
	\begin{Lemme} \label{lem:iota_morphism}
		Let $t$ and $s$ having the same combs representations as equation~\eqref{eq:tree_shapes}.
		The map $\iota$ is surjective and for any $\sigma\in\QSh(k,l)$ and $t,s\in\Tridend(\N^*)$:
		\[
		\iota(\sigma(t,s))=\sigma(\iota(t),\iota(s)).
		\]
	\end{Lemme}
	\begin{proof}
			Let $t\in\Tree(\N^*)$.
			By definition of $\Tree(\N^*)$ the decoration on any $v\in\nu(t)$ is greater than its number of angles. We define $t'$ an element of $\Tridend(\N^*)$ with the same tree structure such that for any $v\in\nu(t)$ the angles of $v$ are decorated from left to right by a partition of the decoration of $v$ of length the number of its angles. Hence, $\iota(t')=t$ so $\iota$ is surjective.
			
			By construction of the action, $\iota$ sends the concatenation of the decorations of angles in $\Tridend(\N^*)$ to the sum of decorations in $\Tree(\N^*)$. Thus, it fulfils the equation required to achieve the proof of the lemma.
	\end{proof}
	As a consequence from this lemma, we deduce that $\Tree(\N^*)$ has a tridendriform structure inherited from the one of $\Tridend(\N^*)$.
	\begin{Prop} \label{prop:tree_tridend}
     For any trees $(t,s)\in\Tree(\N^*)^2$, let us define the products
	 \begin{align*}
		t\prec s\coloneqq \sum_{\substack{\sigma\in\QSh(k,l), \\ \sigma^{-1}(\{1\})=\{1\}}} \sigma(t,s), && t\succ s\coloneqq \sum_{\substack{\sigma\in\QSh(k,l), \\ \sigma^{-1}(\{k+1\})=\{k+1\}}} \sigma(t,s), && t\cdot s\coloneqq \sum_{\substack{\sigma\in\QSh(k,l), \\ \sigma^{-1}(\{1,k+1\})=\{1,k+1\}}} \sigma(t,s).
	\end{align*}
	Then $(\Tree(\N^*),\prec,\succ,.)$ is a tridendriform algebra and $\iota$ is a morphism of tridendriform algebras.
	We will denote the sum of these three products by $*$.
	\end{Prop}
	\begin{proof}
	Let us first show that $\iota$ is an algebra morphism for each of the products $\prec$, $\succ$ and $\cdot$. For any trees $(t,s)\in\Tridend(\N^*)^2$ we have
	\begin{align*}
	 \iota(s\prec t) & =\iota\left(\sum_{\substack{\sigma\in\QSh(k,l), \\ \sigma^{-1}(\{1\})}=\{1\}} \sigma(t,s)\right) \tag*{by definition of $\prec$} \\
	 & = \sum_{\substack{\sigma\in\QSh(k,l), \\ \sigma^{-1}(\{1\})=\{1\}}}\iota\left( \sigma(t,s)\right) \tag*{by linearity of $\iota$} \\
	 & = \sum_{\substack{\sigma\in\QSh(k,l), \\ \sigma^{-1}(\{1\})}=\{1\}}\sigma(\iota(t),\iota(s))\tag*{by lemma~\ref{lem:iota_morphism}} \\
	 & = \iota(t)\prec\iota(s)
	\end{align*}
	by definition of $\prec$ in $\Tree(\N^*)$. We show similarly that $\iota(s\succ t) = \iota(t)\succ\iota(s)$ and $\iota(s\cdot t) = \iota(t)\cdot\iota(s)$.

	     We can now show that the seven equalities of definition~\ref{defi:tridend} hold for the products of the proposition. Take $(r,s,t)\in\Tree(\N^*)^3$, let us show that $(r\prec s)\prec t=r\prec(s* t)$. From the surjectivity of $\iota$, there exist $(r',s',t')\in\Tridend(\N^*)^3$ such that $\iota(r')=r$, $\iota(s')=s$, and $\iota(t')=t$. Then
	     \begin{align*}
	      (r\prec s)\prec t & =  (\iota(r')\prec \iota(s'))\prec \iota(t') \\
	      & = \iota\left((r'\prec s')\prec t' \right) \tag*{by the previous point} \\
	      & = \iota\left(r'\prec (s'* t') \right) \tag*{since $\Tridend(\N^*)$ is tridendriform} \\
	      & = r\prec (s* t)
	     \end{align*}
	     again since $\iota$ is a morphism for the $\prec$ and $\cdot$ products. The other six equalities are shown in the same way.
	     
	     Then, by lemma~\ref{lem:iota_morphism}, $\iota$ is a morphism of tridendriform algebras.
	    \end{proof}
	    Let us make some computations in the tridendriform structure of $\Tree(\N^*)$:
	    \begin{Eg} \label{ex:prod_Tree}
	    Working with low decorations we find
	    \begin{equation*}
	     \labYY{2}{1}{1}\prec \labY{2} = \labYI{2}{1}{1}{2} + \labIY{2}{2}{1}{1}+\labIpY{2}{3}{1}.
	    \end{equation*}
	    For the other products we have
	    \begin{equation*}
	     \labYY{2}{1}{1}\succ \labY{2} = \labYgI{2}{2}{1}{1},\qquad \labYY{2}{1}{1}\cdot \labY{2} = \labYYI{4}{1}{1}.
	    \end{equation*}
	    \end{Eg}
\begin{Rq}
	By the freeness property of $\Tridend(\N^*)$, it shows that $\iota$ is the unique tridendriform morphism such that the diagram~\ref{diag:iota_map} commutes with $i_1$ and $i_2$ are the natural injections.
	\begin{figure}[h]
		\centering
		\begin{tikzcd}
			\Tridend(\N^*) \arrow[r,two heads, "\iota"] & \Tree(\N^*) \\
			& \N^* \arrow[u, hook,"i_1"] \arrow[lu, hook, "i_2"]
		\end{tikzcd}
		\caption{The $\iota$ map}
		\label{diag:iota_map}
	\end{figure}
	Here, $i_2$ is the map $i$ of theorem~\ref{thm:univ_prop}.
\end{Rq}
Notice that the properties of $\Tree(\N^*)$ can be generalised to $\Tree(\Omega)$ for any set $\Omega$ with a commutative monoid structure. We will not need this level of generality in this paper and we omit these statements, whose proofs can be copied mutatis mutandis for those wrote above.

	\subsection{\AZVs{}}

In the literature~\cite{Manchon_16,Ya20}, we can find a definition of an arborification of the \MZV{}. We introduce a similar concept adapted to the trees we are considering:
\begin{defi}[Arborification of the integral \MZVs{}]\label{def:azvs}
	We define the following linear map $\zeta^T:\Treecv \rightarrow  \R$ such that for any tree $t$ of $\Treecv$, we define:
	\[
	\zeta^T(t)=\sum_{\boldsymbol{k}\in D_t} \prod_{v\in\nu(t)} \frac{1}{k_v^{n_v}},
	\]
	where $D_t$ is the set defined by:
	\[
	\left\{ (k_v)_{v\in\nu(t)}\in\N^{|\nu(t)|} \,\middle|\, \forall (v,w)\in \nu(t)^2 , k_v<k_w \text{ if } w\leq v  \right\},
	\] 
	where for any $(u,v)\in\nu(t)^2, u\leq v$ reading the tree as a Hasse diagram where the root is the minimum of the poset. Here, $\Treecv$ is the subset of $\Tree(\N^*)$ where the series is convergent, and will be specified below.
\end{defi}
Notice that in~\cite{Manchon_16,Ya20,clavier2020double}, these \AZVs{} are defined on non-planar, not necessarily Schroeder, forests. The requirement that internal vertices have at least two direct descendants does not restraint our results since one can simply add spurious leaves related to internal vertices without changing the values of the associated arborified zeta value. Working with planar trees does not change the series, nor their properties as it can be readily checked.
\begin{Eg}For instance:
	\[
	\zeta^T\left(\alabelledtree{2}{2}{1}\right)= \sum_{\substack{0< k_2 < k_1 \\ 0< k_3 <k_1}} \frac{1}{k_1^2}\frac{1}{k_2^2}\frac{1}{k_3}.
	\]
\end{Eg}
\begin{Not}\label{Not:grafting_B}
	Let $k\in\N^*$ and $t_1,\dots,t_k$ be $k$ trees of $\Tree(\N^*)$. Let $n\in\N^*$. We denote the tree with root decorated by $n$ which has $t_1$ to $t_k$ for sons from right to left with:
	\[
	\B^+_n(t_1,\dots,t_k).
	\]
\end{Not}
Moreover, there is also the known map:
\begin{defi}\label{defi:flat_map}
	We define the map $\flaten:\Tree(\N^*)\rightarrow \W$ by $\flaten(|)=\emptyset$ and for any ${\B^+_n(t_1,\dots,t_k)=t}$ with $n\in\N^*$ by:
	\[
	\flaten(t)=n\cdot(\flaten(t_1)\qshuffle\dots \qshuffle \flaten(t_k)).
	\]
\end{defi}
Then we have the following properties of $\zeta^T$.
\begin{thm}[\cite{clavier2020double}] \label{thm:knowm_AZV}
 A tree $t\in\Tree(\N^*)$ lies in $\Tree(\N^*)^{\rm conv}$ if, and only if $t=|$ or its root is not decorated by $1$.
 
 Then $\flaten(\Tree(\N^*)^{\rm conv})=\Wcv$ and $\zeta^T$ factorises through the flattening: $\zeta^T=\zeta\circ\flaten$.
\end{thm}
Finally, the map flat have an important property for our context that we will later use.
\begin{Lemme} \label{lem:flat_tridend}
	The map $\flaten:\Tree(\N^*)\rightarrow \W$ is a morphism of tridendriform algebras. 
\end{Lemme}
\begin{proof}
	We prove that $\flaten$ is a morphism of tridendriform algebras with an induction over $v$ the sum of the numbers of internal vertices of the trees we are multiplying.
	\begin{description}
		\item[Initialisation:] $v=2$, we have for any $n,m\in\N^*$:
			\allowdisplaybreaks[3]
		\begin{align*}
		\flaten\left(\labY{n} \prec \labY{m} \right)&=\flaten\left(\labbalaisd{n}{m} \right) =n\flaten\left(\labY{m} \right) = nm=n\prec m, \\
			\flaten\left(\labY{n} \succ \labY{m} \right)&=\flaten\left(\labbalaisg{m}{n} \right) =m\flaten\left(\labY{n} \right) =mn= n\succ m, \\
			\flaten\left(\labY{n}\cdot \labY{m}\right)&=\flaten\left(\raisebox{-0.4\height}{\begin{tikzpicture}[line cap=round,line join=round,>=triangle 45,x=0.4cm,y=0.4cm]
					\draw [line width=.5pt] (0.,0.)-- (0.,2.);
					\draw [line width=.5pt] (0.,1.)-- (-1.,2.);
					\draw [line width=.5pt] (0.,1.)-- (1.,2.);
					\draw[right] (0,1) node {\footnotesize{n+m}}; 
			\end{tikzpicture}}\right)=n+m=n\cdot m.	
		\end{align*}
		\allowdisplaybreaks[0]
		\item[Heredity:]suppose that there exists $v$ such that for all pair of trees $(t_1,t_2),{|\nu(t_1)|+|\nu(t_2)|\leq v}, \flaten$ behaves like a tridendriform morphism. 
		Consider $t,s$ two elements of $\Tree(\N^*)$ with ${|\nu(t)|+|\nu(s)|=v+1}$. Let us show that $\flaten(t\ltimes s)=\flaten(t)\ltimes \flaten(s)$ for $\ltimes\in\{\prec,\succ,\cdot\}$. If $|\nu(t)|=0$ or $|\nu(s)|=0$, the result trivially holds. Otherwise, we can write 
		\begin{equation*} 
		 t=\B^+_n(t_1,\cdots,t_k),\quad s=\B^+_m(s_1,\cdots,s_l).
		\end{equation*}
		Then, using the inductive descriptions of the tridendriform products from lemma~\ref{Lem:induction_product} (which hold for the tridendriform products of $\Tree(\N^*)$ by lemma \ref{lem:iota_morphism}) we have:
		\begin{align*}
			\flaten(t\prec s)&=\flaten\left( \B^+_n\left( t_1, \dots , (t_k*s)\right) \right) \\
			&= n\left( \flaten\left( t_1 \right) \cshuffle \cdots \cshuffle \flaten\left( t_k* s \right) \right) \tag*{by definition of $\flaten$} \\
			&=n\left( \flaten\left( t_1 \right)\cshuffle \cdots\cshuffle \flaten\left( t_k \right)\cshuffle \flaten\left( s \right) \right)
		\end{align*} from the associativity of $\cshuffle$ and the induction hypothesis. Hence:
		\begin{align*}
			\flaten(t\prec s)&= \left[ n\left(\flaten\left( t_1 \right)\cshuffle \cdots \cshuffle \flaten\left( t_k \right) \right) \right]\prec \flaten\left( s\right) \\
			&=\flaten(t)\prec \flaten(s).
		\end{align*}
		Moreover, using the same ideas:
		\begin{align*}
			\flaten\left( t\succ s \right)
			&=m\left( \flaten\left( t* s_1 \right)\cshuffle \cdots \cshuffle \flaten\left(s_l \right) \right) \\
			&=\flaten(s)\succ \left[m\left(\flaten\left( s_1\right)\cshuffle \cdots \cshuffle \flaten\left( s_l\right) \right)\right] \\
			&=\flaten\left(s \right)\succ \flaten\left( t\right).
		\end{align*}
		Finally:
		\begin{align*}
			\flaten(t\cdot s)&= \flaten\left(\B^+_{n+m}(t_1,\dots, t_k*s_1, s_2,\dots, s_l)\right) \\
			&=(n+m) \left( \flaten\left(t_1\right)\cshuffle \cdots \cshuffle \flaten\left(t_k* s_1\right)\cshuffle\cdots \cshuffle  \flaten\left(s_l\right)  \right) \\
			&=(n+m) \left( \flaten\left(t_1\right)\cshuffle \cdots \cshuffle \flaten\left(t_k\right)\cshuffle \flaten\left(s_1\right)\cshuffle\cdots \cshuffle  \flaten\left(s_l\right)\right) \\
			&=\left[ n \left( \flaten\left(t_1\right)\cshuffle \cdots \cshuffle \flaten\left(t_k\right)\right)\right] \cdot \left[m\flaten\left(s_1\right)\cshuffle\cdots \cshuffle  \flaten\left(s_l\right)\right] \\
			&=\flaten(t)\cdot \flaten(s).
		\end{align*} 
	\end{description}
	Hence, by the induction principle, it shows that $\flaten$ is a tridendriform morphism between $\Tree(\N^*)$ and $\W$.
\end{proof}

\subsection{Relating the Zetas}

Let us recall the following maps, which were respectively introduced in theorem~\ref{thm:univ_prop} and definition~\ref{def:triden_zeta}:
\begin{align*}
		i:\left\{ \begin{array}{rcl}
			\N^* &\rightarrow & \Tridend(\N^*), \\
			n & \mapsto & \labY{n}, 
		\end{array}\right. && \sum:\left\lbrace\begin{array}{rcl}
			\N^* & \rightarrow & \Suum_{\N^*}, \\
			n &\mapsto & \displaystyle\sum_{1\leq k} f_n.
		\end{array} \right.
	\end{align*}
Let us also denote by $j$ the natural inclusion from $\N^*$ into $\W$. Then we recall the maps of tridendriform algebras: $\zeta^{\rm Tri}:\Tridend(\N^*)\longrightarrow\Suum_{\N^*}$ and $\Psi:\Tridend(\N^*)\longrightarrow\W$ whose existence and unicity is given by the universal property of $\Tridend(\N^*)$ (theorem~\ref{thm:univ_prop}) as we can see in diagrams~\ref{diag:1tridendfreeness} and~\ref{diag:2tridendfreeness}. 
\begin{figure}[ht]
	\centering 
	\begin{subfigure}[t]{0.45\linewidth}
		\begin{tikzcd}
			\N^* \arrow[rd, "j"] \arrow[r, "i"] & \Tridend(\N^*) \arrow[d ,dashed, "\exists!\,\Psi"] \\
			& \W
		\end{tikzcd}
		\caption{First freeness property of $\Tridend(\N^*)$}
		\label{diag:1tridendfreeness}
	\end{subfigure}
	\begin{subfigure}[t]{0.45\linewidth}
		\begin{tikzcd}
			\N^* \arrow[rd, "\sum"] \arrow[r, "i"] & \Tridend(\N^*) \arrow[d ,dashed, "\exists!\,\zeta^{\rm Tri}"] \\
			& \Suum_{\N^*}
		\end{tikzcd}
		\caption{Second freeness property of $\Tridend(\N^*)$}
		\label{diag:2tridendfreeness}
	\end{subfigure}
\label{diag:freeness}
\caption{Freeness diagrams}
\end{figure}
These tridendriform maps and algebras are related.
\begin{thm} \label{thm:zeta_tridend}
 The following diagram commutes: 
 \begin{equation}\label{diag:thm_tridend} 
   \xymatrix{
    \Tree(\N^*)\ar@{->}[rd]_\flaten & \Tridend(\N^*) \ar@{->>}[l]_\iota \ar@{->}[d]^\Psi \ar@{->}[r]^-{\zeta^{\rm Tri}} & \calS_{\N^*} \\ 
   & \W \ar@{->}[ur]_{\zeta_{\rm FS}}  &
 }.
 \end{equation}
\end{thm}
\begin{proof}
 We know that $\iota$ and flat are tridendriform algebras morphisms (proposition~\ref{prop:tree_tridend} and lemma~\ref{lem:flat_tridend} respectively). Thus $\flaten\circ\iota:\Tridend(\N^*)\longrightarrow\W$ is a morphism of tridendriform algebras. But from the definition of maps $i$ and $j$ in the diagram~\ref{diag:1tridendfreeness}, replacing $\Psi$ by $\flaten\circ \iota$ also commutes. Thus, by the unicity of $\Psi$ in this diagram, we have $\Psi=\flaten\circ\iota$.
 
 The same argument also holds for the commutation of the second triangle since ${\zeta_{\FS}:\W\longrightarrow\calS_{\N^*}}$ is a morphism of tridendriform algebras (proposition~\ref{Lem:zeta_morph_tridend}) as well as $\Psi:\Tridend(\N^*)\longrightarrow\calW_{\N^*}$ by definition. Furthermore, replacing $\zeta^{\rm Tri}$ by $\Psi \circ \zeta_{\FS}$ in diagram~\ref{diag:2tridendfreeness} also commutes. So $\zeta^{\rm Tri}=\zeta_{\rm FS}\circ\Psi$ by unicity of $\zeta^{\rm Tri}$.
\end{proof}
Let us now define convergent elements of $\Tridend(\N^*)$.
\begin{defi}\label{def:Schroeder_tree_convergent}
 Let $t\in\Tridend(\N^*)$ be a decorated Schroeder tree. Then $t$ is \emph{convergent} if $t=|$ or the root of $t$ does not have a unique angle decorated by $1$. Let $\Tridend(\N^*)^{\rm conv}$ be the subvector space of $\Tridend(\N^*)$ generated by convergent trees.
\end{defi}
Then specializing the diagram of theorem~\ref{thm:zeta_tridend} to convergent trees we obtain
\begin{Cor} \label{coro:relation_zetas}
 The following diagram commutes:
 \begin{equation*}
   \xymatrix{
    \Tree(\N^*)^{\rm conv} \ar@{->}[rd]_\flaten \ar@{->}@/_1pc/[rdd]_{\zeta^T} & \Tridend(\N^*)^{\rm conv} \ar@{->>}[l]_\iota \ar@{->}[d]^\Psi \ar@{->}[r]^-{\zeta^{\rm Tri}} & \calS_{\N^*}^{\rm conv}  \ar@{->}@/^1pc/[ldd]^{\rm ev}\\ 
   & \Wcv \ar@{->}[ur]_{\zeta_{\rm FS}} \ar@{->}[d]^\zeta  & \\
   & \R & 
 }.
 \end{equation*}
\end{Cor}
\begin{proof}
 From the definitions of the maps and the spaces in this diagram, it is clear that the maps send convergent spaces to convergent spaces. The two upper triangles are a special case of theorem~\ref{thm:zeta_tridend}. The left lower triangle is theorem~\ref{thm:knowm_AZV}. The right lower triangle is a direct consequence of the definition of $\zeta_{\rm FS}$ (definition~\ref{def:zeta_FS}).
\end{proof}
In particular, we obtain that tridendriform zeta values are \AZVs{}, and also linear combinations (with integer coefficients) of \MZVs{}, given by
\begin{equation*}
 \forall t\in\Tridend(\N^*)^{\rm conv},~\ev\circ\zeta^{\rm Tri}(t)=\zeta^T(\iota(t))= \zeta\circ\flaten\circ\iota(t).
\end{equation*}
Finally, notice that this construction implies a new property for \AZVs{}.
\begin{Cor} \label{coro:AZV_alg_morph}
The map $\zeta^T:\Treecv\rightarrow \R$ is an algebra morphism for the quasi-shuffle product of $\Tree(\N^*)$: for any couple of trees $(t_1,t_2)\in\left(\Treecv\right)^2$ we have 
 \[
	\zeta^T(t_1* t_2)=\zeta^T(t_1)\zeta
	^T(t_2).
	\]
\end{Cor}

\begin{proof}
 It is clear from the definition that $\Treecv$ is a tridendriform subalgebra of $\Tree(\N^*)$. Then for any couple of trees $(t_1,t_2)\in\left(\Treecv\right)^2$ we have 
 \begin{align*}
  \zeta^T(t_1* t_2) & = \zeta\circ\flaten(t_1* t_2)\tag*{by corollary~\ref{coro:relation_zetas}} \\
  & = \zeta\circ\left(\flaten(t_1)\cshuffle \flaten(t_2)\right) \tag*{by lemma~\ref{lem:flat_tridend}} \\
  & = \zeta\circ\flaten(t_1)\zeta\circ\flaten(t_2)\tag*{since $\zeta$ is an algebra morphism for $\cshuffle$} \\
  &=\zeta^T(t_1)\zeta
	^T(t_2)\tag*{by corollary~\ref{coro:relation_zetas}.} 
 \end{align*}
 Hence, we have proved the theorem. 
\end{proof}
Let us write down an explicit example of new relations amongst arborified zeta values using the computation of example~\ref{ex:prod_Tree}
\begin{Eg} 
	 We have
 \begin{align*}
  &\zeta^T\left(\labYY{2}{1}{1}\right)\zeta^T\left(\labY{2}\right) \\ 
  = & \zeta^T\left(\labYI{2}{1}{1}{2}\right) + \zeta^T\left(\labIY{2}{2}{1}{1}\right)+\zeta^T\left(\labIpY{2}{3}{1}\right) + \zeta^T\left(\labYgI{2}{2}{1}{1}\right) + \zeta^T\left(\labYYI{4}{1}{1}\right).
 \end{align*}
This relation can be checked using corollary~\ref{coro:relation_zetas}. It allows us to reformulate this equation as
\begin{align*}
	\left(2\zeta(2,1,1)+\zeta(2,2)\right)\zeta(2)&=	\zeta(2,1,2,1) +  2\zeta(2,1,1,2)+\zeta(2,1,3)+\zeta(2,2,2) \\
	+ & 2\zeta(2,2,1,1)+\zeta(2,1,2,1)+\zeta(2,3,1)+\zeta(2,2,2) \\
	+ & \zeta(2,3,1)+\zeta(2,1,3)+ \zeta(2,4) \\ 
  & +2\zeta(2,2,1,1)+\zeta(2,2,2)+2\zeta(4,1,1)+\zeta(4,2)
\end{align*}
which can be checked using the first equality in equation~\eqref{eq:shuffle_stuffle_zeta}.
\end{Eg}

We have solved the issue we sat ourselves upon and related our tridendriform zeta values to the usual \MZVs{} and their arborified counterparts. As a consequence of the structures underlining our approach, we found an associative product for Schroeder trees such that the arborified \MZV{} is a morphism for this product. In the next section, we are doing the same construction with the integral representation of \MZVs{} using a dendriform structure. It corresponds to a tridendriform structure where the $\cdot$ product is equal to $0$.
The proofs are similar and easier in this part. So, we will give less details in order to save space.
	
	\section{Dendriform algebra of formal integrals}\label{sec:dendri}
	
		We first built a tridendriform algebra to recover the arborified version of \MZVs{}. We now perform a quite similar construction to section~\ref{sec:tridend} with formal integrals. One can see the previous work as a specification of this one modifying the Lebesgue measure with another one with Radon-Nikodym decomposition  $\nu+\mu$ where $\mu\neq 0$ and $\mu$ and $\lambda$ are singular. As a consequence, the "diagonal" terms do not vanish any more.
			
	\subsection{Formal integrals}
		
	\begin{Not}
		Let $n\in\N^*.$
		We denote by $\lambda$ the \emph{Lebesgue measure over} $\R^n$. We denote by $\Borel(\R^n)$ the \emph{borelian set} of $\R^n$ for $\lambda$.
		We also denote by $\Neg(\R^n)$ the subset of $\Borel(\R^n)$ containing all negligible sets for the Lebesgue measure of $\R^n$. 
		Let $n\in\N^*$ and set $\B_{\Int}(\R^n)\coloneqq\faktor{\Borel(\R^n)}{\Neg(\R^n)}$. Hence, $\B_ {\Int}(\R^n)$ is our notation for the set of borelian set up to sets of measure $0$.
	\end{Not}
	\begin{Rq}
		One can take any measure to define formal integrals. For the sake of simplicity we take the Lebesgue measure since it is largely enough for our purposes.
	\end{Rq}

	\begin{defi}\label{defi:formal_integrals}
		For any $n\in\N^*,$ we consider the formal vector space $\K\B_{\Int}(\R^n)$ whose basis is $\B_{\Int}(\R^n)$ and we extend linearly the definition of the union $\cup$. We define $I$ the ideal of the algebra $(\K\B_{\Int}(\R^n),+,\cup)$ generated by the set of elements:
		\[
		\left\lbrace A\cup B-A-B+A\cap B \,\middle|\, A,B\in\B_{\Int}(\R^n)\right\rbrace.
		\]
		We finally set
		\[
		\Bint(\R^n)=\faktor{\K\B_{\Int}(\R^n)}{I}.
		\]	
		Let $\Omega$ be a set, $w\in \Omega^{\star}$ and a map $G\in(\R \rightarrow \R)^\Omega$. We denote for any $\omega\in\Omega, g_\omega\coloneqq G(\omega)$. Set $G_{w}:\R^{k}\rightarrow \R$ with $w=w_1\dots w_k$ the function defined by:
		\[
		G_w(x_1,\dots, x_k)\coloneqq \prod_{i=1}^{k} g_{w_i}(x_i).
		\]
		We call a \emph{formal integral} an element $(B\otimes w, G)$  of
		\begin{equation}
		\left\langle \Bint(\R^n)\otimes w  \,\middle|\, w\in\calW_{\Omega}, n=\ell(w) \right\rangle \times (\R\rightarrow \R)^{\Omega}. \label{eq:formal_integrals_building}
		\end{equation}
		A formal integral $(B\otimes w,G)$ with $w=w_1\cdots w_n$ will be written
			\[
			\int_A \left(\prod_{i=1}^{k} g_{w_i} \right)\dt_1\dots \dt_k.
			\]
			The omission of the variable in which the function is taken will distinguish formal integrals from usual integrals.
			We will denote $\FI_{\Omega}$ the set of definite formal integrals over $\Omega$.
	\end{defi}
	\begin{Not}
		Following notation~\ref{Not:cart}, we will denote a formal integral by $B\otimes w\times G$ .
	\end{Not}

	\begin{Rq}
	The notation of formal integrals is made such that a specification of $F$ gives a concrete integral to compute.	
	\end{Rq}
	As said in the previous section, our aim is to consider formal objects before any evaluation of the integral.
	
	\begin{Eg}
		In this section, we will mainly apply this construction with $\Omega=\{x,y\}$ and $G$ is the map:
		\[
		G: \left\lbrace\begin{aligned}
			&x\mapsto \left( g_x:\left\lbrace\begin{array}{rcl}
				\interoo{0 1} & \rightarrow & \R, \\
				t & \mapsto & \frac{1}{t},  
			\end{array} \right.\right), \\
			&y\rightarrow \left( g_y:\left\lbrace\begin{array}{rcl}
				\interoo{0 1} & \rightarrow & \R, \\
				t & \mapsto & \frac{1}{1-t}.
			\end{array} \right. \right)
		\end{aligned} \right.
		\]
		Hence a formal version of the following \MZV{} 
		\[
		\int_{\substack{0< t_2< t_1< 1, \\ 0< t_3 < t_1< 1}} \frac{\dt_1}{t_1}\frac{\dt_2}{1-t_2}\frac{\dt_3}{1-t_3}
		\]
		is given by:
		\[
		\{0< t_2 < t_2 < 1 \text{ or } 0< t_3 < t_1< 1\} \otimes xyy\times G.
		\]
		\end{Eg}
	\begin{Rq}
		Note that in the previous example omitted the limit step. One should get the formal version for any $0< a< b< 1$ then taking the limit $a\to 0$ and $b\to 1$. For sake of simplicity, we will always omit this step later on.
	\end{Rq}

	The previous construction enables the following manipulations of formal integrals, which mimic the properties of usual integrals. 
	\begin{Prop}\label{prop:formal_int_properties}
	Let $n\in \N^*$. Let $\Omega$ be a set, $A,B\in\Bint(\R^n)$, $w\in\calW_\Omega$ with $\ell(w)=n$, and $G\in(\R\rightarrow \R)^{\Omega}$.
		\begin{itemize}
			\item If $(A\cup B)\otimes w\times G$ is a formal integral over $\Omega$, then so are $A\otimes w\times G$ and $B\otimes w\times G$. Furthermore we have
			\[
			\int_{A\cup B} \left(\prod_{i=1}^n g_{\omega_i} \right)\prod_{i=1}^n\dt_i=\int_A \left( \prod_{i=1}^n g_{\omega_i} \right)\prod_{i=1}^n\dt_i+\int_B \left( \prod_{i=1}^n g_{\omega_i} \right)\prod_{i=1}^n\dt_i-\int_{A\cap B} \left( \prod_{i=1}^n g_{\omega_i}\right) \prod_{i=1}^n\dt_i.
			\]
			\item Let $w_1=\omega_1^{(1)}\dots \omega_1^{(k)}, w_2=\omega_2^{(1)}\dots \omega_2^{(k)}$ and $A\in\Bint(\R^k)$. Then:
			\begin{align*}
				&\int_A \left( \prod_{i=1}^k g_{\omega_1^{(i)}}+\prod_{j=1}^k g_{\omega_2^{(j)}} \right)\prod_{i=1}^{k} \dt_i 
				=\int_A \left( \prod_{i=1}^k g_{\omega_1^{(i)}}\right)\prod_{i=1}^{k}\dt_i
				+\int_A \left( \prod_{j=1}^n g_{\omega_2^{(j)}}\right)\prod_{j=1}^{k}\dt_j.
			\end{align*}
		\end{itemize}
	\end{Prop}
	\begin{proof}
		A similar proof strategy to proposition~\ref{prop:formal_series_properties} applies here.
	\end{proof}
	
	\subsection{The dendriform algebra of formal integrals}	
	
	In this subsection, we will mainly focus on the case $\Omega=\{x,y\}$ and:
		\begin{equation}
			G: \left\lbrace\begin{aligned}
		&x\mapsto \left( g_x:\left\lbrace\begin{array}{rcl}
			\interoo{0 1} & \rightarrow & \R, \\
			t & \mapsto & \frac{1}{t},  
		\end{array} \right.\right), \\
		&y\rightarrow \left( g_y:\left\lbrace\begin{array}{rcl}
			\interoo{0 1} & \rightarrow & \R, \\
			t & \mapsto & \frac{1}{1-t}.
		\end{array} \right. \right) 
	\end{aligned} \right. \label{eq:G_function}
		\end{equation}
	
	In order to lighten the notations, we will omit the dependencies on $G$ in the remaining of this section. Note that some results are still true other appropriate $G$ and $\Omega$.
	
	\begin{defi} \label{def:formal_Chen_int}
	 Define the space of \emph{formal Chen integrals over $\Omega$} as the subspace of formal integrals with $G$ fixed of $\A_{\Omega}$ generated by
		\[
		\bigcup_{r\geq 0}\bigcup_{w\in \calW_{\Omega,r}} \left\{ \int_{0< t_1 < \dots < t_r< 1} \left(\prod_{i=1}^r g_{\omega_i}\right) \prod_{i=1}^r \dt_i \right\}
		\]
		where we set $w=\omega_1\dots\omega_r$ as usual. In fact, $\calA_{\Omega}$ inherits the vector space structure of left space in equation~\eqref{eq:formal_integrals_building}. Hence, $\calA_{\Omega}$ is the set of formal integrals needed to build \MZVs{} in their integral version (equation~\eqref{Prop:integral_representation_classic}).
	\end{defi}
It is well known, from standard analysis tools, that a formal integral $A\otimes w \times G\in\A_{\Omega}$ is equivalent to a convergent integral provided that the first letter of the word is an $x$ and the last one is a $y$. Then, the following map will encode all the analysis job to get any real value.
	\begin{defi}[Evaluation map]\label{defi:ev_map_dend}
		A \emph{formal integral} $A\otimes w \times G$ is said to be \emph{convergent} if $w=\ew$ or if $w$ begins with a $x$ and ends with a $y$ (see definition~\ref{def:conv_words}). We will denote the space of convergent formal integrals by $\FI^{\rm conv}_{\Omega}$. We define the \emph{evaluation map}, denoted $\ev$, as follows:
		\begin{equation*}
			\ev_{\Omega}:\left\lbrace\begin{array}{rcl}
				\FI^{\rm conv}_{\Omega} & \rightarrow & \R, \\
				\displaystyle \int_A \left( \prod_{i=1}^k f_{\omega_i} \right) \prod_{i=1}^k\dt_i & \mapsto & \displaystyle\int_A\left( \prod_{i=1}^k f_{\omega_i}(t_i) \right)\prod_{i=1}^k\dt_i.
			\end{array} \right.
		\end{equation*}
	\end{defi}	
	\begin{Rq}
	This evaluation map satisfies similar properties to remark~\ref{Rq:not_linear}.	
	\end{Rq}

	Using Chen's lemma for iterated integrals~\cite{chen1977iterated}, we introduce the unique shuffle product $\shuffle$ on formal integrals such that the evaluation map is a morphism for the multiplicative structure of $\R$.
	\begin{defi}\label{def:dend_structrure}
		We introduce two new operators $\prec$ and $\succ$ defined for any $\alpha,\beta\in\A_{\Omega}$ written as follows:
		\begin{align*}
			\alpha=\int_{0<t_1<\dots <t_k<1} \prod_{i=1}^{k} f_{\omega_i}\dt_i, 
			&& \beta=\int_{0<t_{k+1}<\dots <t_{k+l}<1} \prod_{i=k+1}^{k+l} f_{\omega_i}\dt_i,
		\end{align*}
		by:
		\begin{align*}
			\alpha\prec\beta&\coloneqq\sum_{\substack{\sigma\in\Sh(k,l) \\ \sigma^{-1}(\{1\})=\{1\}}}\int_{0<t_{1}<\dots <t_{k+l}<1} \prod_{i=1}^{k+l} f_{\omega_{\sigma^{-1}(i)}}\dt_i, \\
			\alpha\succ \beta&\coloneqq\sum_{\substack{\sigma\in\Sh(k,l) \\ \sigma^{-1}(\{1\})=\{k+1\}}}\int_{0<t_{1}<\dots <t_{k+l}<1} \prod_{i=1}^{k+l} f_{\omega_{\sigma^{-1}(i)}}\dt_i.
		\end{align*}
		Moreover, we have $\alpha\shuffle \beta \coloneqq \alpha\prec\beta+\alpha\succ \beta$.
		\end{defi}
		Then, we have analogous results to the tridendriform part.
	\begin{Prop}  \label{prop:chen_dend_struct}
		With this structure $(\A_{\Omega},\prec,\succ)$ is a dendriform algebra.
	\end{Prop}
	\begin{proof}
	A similar proof strategy to proposition~\ref{Prop:tridend_Suum} proves the proposition.
	\end{proof}
	\begin{Prop}\label{lem:gen_dend}
		Let $\Omega$ be a set and $G:\Omega \rightarrow \R^{\R}$. Then, $(\A_{\Omega},\prec,\succ)$ is generated as a dendriform algebra by the following set:
		\[
		\left\{  \int_{0}^1 g_{\omega}\dt\right\}_{\omega\in\Omega}.
		\]
	\end{Prop}
	\begin{proof}
		The proof is similar to proposition~\ref{lem:gen_tridend}.
	\end{proof}
	Finally, we have the dendriform version of proposition~\ref{prop:ev_alg_morph}, which is proved in exactly the same fashion:
	\begin{Prop}
	 $\FI^{\rm conv}_\Omega$ is a dendriform subalgebra of $\A_\Omega$ and ev is an algebra morphism for the product $\shuffle$.
	\end{Prop}

	 \begin{defi}[\MZV{} integral representation] \label{def:formal_MZV}
		We define the following linear map ${\zeta_{\FI}:\Wxycv \rightarrow \A_{\{x,y\}}}$ for any word $w=\omega_1\dots \omega_k\in \Wxy$ by:
		\[
		\zeta_{\FI}(w)=\int_{0< t_1 < \dots < t_k < 1}\left( \prod_{i=1}^k g_{\omega_i}\right)\prod_{i=1}^k \dt_i
		\]
		and $\zeta_{\FI}(\emptyset):=\emptyset\otimes\emptyset\times G$. 
		An element in the image of $\zeta_{\FI}$ is called a \emph{formal integral \MZV{}}.
	\end{defi}
	Hence, the \emph{classical integral \MZV{}} of proposition~\ref{Prop:integral_representation_classic} satisfies:
	\[
	\zeta_{\Int}=\ev\circ \zeta_{\FI}.
	\] 
	
As a first advantage of the formal integral version of \MZVs{}, we have that $\zeta_{\FI}$ is a morphism of dendriform algebra. This is the dendriform counterpart to proposition \ref{Lem:zeta_morph_tridend}.
	\begin{Prop} \label{prop:formal_MZV_dend_map}
		The map $\zeta_{\FI}:\Wxy\rightarrow \A_{\{x,y\}}$ is a morphism of dendriform algebras.
	\end{Prop}
	\begin{proof}
		Using the fact that $\prec=\succ^{\textrm{op}}$, it is enough to prove that $\zeta_{\FI}$ is a morphism for $\prec$. Let $v,u$ be two elements of $\Wxy$ such that $v=\omega_1\dots \omega_k$ and $u=\omega_{k+1}\dots \omega_{k+l}$. Then:
		\begin{align*}
			\zeta_{\FI}(v\prec u)&=\sum_{\substack{\sigma\in\Sh(k,l) \\ \sigma^{-1}(1)=1}} \zeta_{\FI}\left(\omega_{\sigma^{-1}(1)}\dots \omega_{\sigma^{-1}(k+l)} \right) \\
			&=\sum_{\substack{\sigma\in\Sh(k,l) \\ \sigma^{-1}(1)=1}} \int_{0<t_{1}<\dots <t_{k+l}<1} \prod_{i=1}^{k+l} f_{\omega_{\sigma^{-1}(i)}}\dt_i \\
			&=\left( \int_{0< t_1< \dots < t_k< 1} \prod_{i=1}^k f_{i}\dt_i \right)\prec \left( \int_{0< t_{k+1}< \dots < t_{k+l}< 1} \prod_{j=k+1}^{k+l} f_{j}\dt_j \right) \\
			&=\zeta_{\FI}(v)\prec \zeta_{\FI}(u).
		\end{align*}
		This concludes the proof.
	\end{proof}
	This proposition indicates that the property of integral representation of \MZVs{} being a morphism for the shuffle product of $\Wxy$ comes from a deeper dendriform structure. Studying the link between the evaluation map and the dendriform structure is beyond the scope of this article since we are not aiming at deriving new relations among \MZVs{} and is therefore left for future research.
	
	\section{\MZVs{} and the dendriform structure}\label{sec:six}
	
	The construction in this section is similar to the one in the tridendriform case section~\ref{sec:three}. In this particular case, the construction is simpler than in the previous one.
	We introduce the combinatorics of the free dendriform algebra. 
	
	\begin{defi}[binary tree]\label{def:binary_tree}
		A \emph{binary tree} is a \Sch{} tree where all of its internal vertices has exactly two children. We will denote the set of binary trees union $|$ by $\BT$.
		Let us denote by $\BT(\{x,y\})$ the set of binary trees whose angles are decorated by an element of $\{x,y\}$.  
	\end{defi}
	
	\begin{Rq}
		Given a binary tree, a labelling of its angles is equivalent to do a labelling of its internal vertices.  
	\end{Rq}
	
	\begin{Eg} As before, we write $\BT_n$ (resp. $\BT_n(\{x,y\})$) the set of binary trees with $n+1$ leaves (resp. decorated by $\{x,y\}$).
		Here are some examples of binary trees:
		\[
		\BT_0=\{|\}, \quad \BT_1=\left\{\labY{}\right\},\quad \text{and } \BT_2=\left\{\labbalaisd{}{}, \labbalaisg{}{}\right\}.
		\]
		Moreover, we list here all the trees of $\BT(\{x,y\})$ with at most $3$ leaves:
		\begin{gather*}
			\BT_{0}(\{x,y\})=\{|\},\quad \BT_{1}(\{x,y\})=\left\{\labY{x},\labY{y}\right\}, \text{ and } \\
			 \BT_{2}(\{x,y\})=\left\{\labbalaisd{x}{x}, \labbalaisd{y}{y}, \labbalaisd{x}{y}, \labbalaisd{y}{x}, \labbalaisg{x}{x}, \labbalaisg{y}{y}, \labbalaisg{x}{y}, \labbalaisg{y}{x}\right\}.
		\end{gather*}
	\end{Eg}
	We recall the notations for \Sch{} trees that are still valid for binary trees.

\begin{Not}
	Let $t$ be a tree in $\Dend(\Omega)$. We denote by $\nu(t)$ the set of internal vertices of $t$. 
	We also define the map $d:\nu(t)\rightarrow \Omega$ such that $v$ is decorated by $d_v.$
\end{Not}
		
	The free dendriform algebra over $\{x,y\}$  has for underlying vector $\K\BT(\{x,y\})$. The products are defined via an action of the shuffles over the combs representation of these trees, see~\cite{Catoire_23} in the case where $\cdot=0$. The combinatorial description is a simpler version of the one in the free tridendriform algebra involving only \emph{shuffles} introduced in definition~\ref{defi:qsh_sh2}.
	Let $t,s$ be two elements of $\BT(\{x,y\})$ $|$. We see $t$ as a right comb and $s$ as a left comb. In other words, we put:
	\begin{center}
		\begin{align}\label{eq:tree_shape_dend}
			t&=\peignedroitdec{F_1}{F_2}{F_k}{w_1}{w_2}{w_k}
			& \text{ and } &&  s=\peignegauchedec{F_{k+1}}{F_{k+2}}{F_{k+l}}{w_{k+1}}{w_{k+2}}{w_{k+l}},
		\end{align}
	\end{center}
	where for all $i\in\IEM{1}{k+l}, F_i$ is a binary tree and $w_i$ is the label of the interval vertex on which $F_i$ is grafted. Restricting ourselves to shuffles in definition~\ref{def:quasiaction}, we have an action of shuffles on any pair $(t,s).$ As a consequence, we also have the analogous of theorem~\ref{thm:produit} in the dendriform case:
	\begin{thm}[\cite{Catoire_23}]\label{thm:produit_dend}
		Let $(t,s)\in\BT(\{x,y\})^2$ different from $|$ as in equation~\eqref{eq:tree_shape_dend}. Then for any set, $\BT(\{x,y\})$ has the structure of a dendriform algebra given by: 
		\[
		t\star s=\sum_{\sigma\in \Sh(k,l)}\sigma(t,s).
		\]
		Moreover:
		\begin{align*}
			t\prec s=\sum_{\substack{\sigma\in\Sh(k,l) \\ \sigma^{-1}(\{1\})=\{1\}}} \sigma(t,s), &&
			t\succ s=\sum_{\substack{\sigma\in\Sh(k,l) \\ \sigma^{-1}(\{1\})=\{k+1\}}} \sigma(t,s).
		\end{align*}
	\end{thm}
	Moreover, it also fulfils a similar lemma to lemma~\ref{Lem:induction_product}:
	\begin{Lemme}\label{Lem:induction_product_dend}
		Let $t$ and $s$ be two trees of $\BT(\{x,y\})$ different from $|$ with:
		\begin{align*}
			t=t^{(1)}\vee_{\omega_1} \dots \vee_{\omega_{k-1}} t^{(k)}, &&
			s=s^{(1)} \vee_{\mu_1} \dots \vee_{\mu_{l-1}} s^{(l)}.
		\end{align*}
		Then:
		\begin{gather*}
			t\prec s=t^{(1)}\vee_{\omega_1} \dots \vee_{\omega_{k-1}} \left(t^{(k)}\star s\right) \text{ and }
			t \succ s= \left(t \star s^{(1)}\right)\vee_{\mu_1} s^{(2)}\vee_{\mu_2} \dots \vee_{\mu_{l-1}} s^{(l)}.
		\end{gather*}
	\end{Lemme}
		
	The freeness property of $\Dend(\{x,y\})$, similar to the one of $\Tree(\Omega)$ stated in theorem~\ref{thm:univ_prop}, is given below.
		\begin{thm}[\cite{Ronco00,Loday_Ronco1998}] \label{thm:univ_prop_dend}
			The structure $(\Dend(\{x,y\}),\prec,\succ)$ is the free dendriform algebra over $\{x,y\}$: for any dendriform algebra $A$ and any map $L:\{x,y\}\longrightarrow A$, there exists a unique morphism of dendriform algebras $\Phi:\Dend(\{x,y\})\longrightarrow A$ such that the diagram below commutes, with ${i:\{x,y\}\longrightarrow \Dend(\{x,y\})}$ the canonical embedding defined by $i(z)\coloneqq\labY{z}$ for $z\in\{x,y\}$:
		\begin{figure}[ht]
			\centering 
				\begin{tikzcd}
					\{x,y\} \arrow[rd, "L"] \arrow[r, "i"] & \Dend(\{x,y\}) \arrow[d, "\Phi"] \\
					& A
				\end{tikzcd}.
		\end{figure}
	\end{thm}
	 Applying the freeness property of $\Dend(\{x,y\})$ to the dendriform algebra $\calA_{\{x,y\}}$, we obtain dendriform zeta values:
	\begin{defi}\label{def:dend_zeta}
		Let $\mathfrak{G}: \{x,y\} \rightarrow \A_{\{x,y\}}$ be the map defined by $z \mapsto \int_0^1 g_z\dt$ where $g_z$ are chosen in equation~\eqref{eq:G_function}. Then the \emph{dendriform zeta values} is the unique morphism of dendriform algebras $\zeta^{\Dend}:\Dend(\{x,y\})\rightarrow \A_{\{x,y\}}$ whose existence and unicity is given by the universal property of $\Dend(\{x,y\})$ stated in theorem~\ref{thm:univ_prop_dend} for $L=\mathfrak{G}$.
	\end{defi}
	Let us recall that $\A_{\{x,y\}}$ is a space of \emph{formal} integrals, as the lack of the integration parameter in $\int_0^1 g_z\dt$ indicates. Thus, the reader should not be concerned with the fact that the integrals $\int_0^1 g_z(t)\dt$ do not converge.
	
	Now, by definition, $\zeta^{\Dend}$ is a morphism of dendriform algebras, therefore it is also a morphism for the shuffle product $\star$ in $\BT(\{x,y\})$ which is explicitly described in theorem~\ref{thm:produit_dend}:
	\[
	\forall (t,s)\in\BT(\{x,y\}), \zeta^{\Dend}(t\star s)=\zeta^{\Dend}(t)\cdot\zeta^{\Dend}(s).
	\]
	Hence, we have build (once again, after using the evaluation map of definition~\ref{defi:ev_map_dend}) a new generalisation of integral \MZVs{} which is a morphism for a generalisation of the shuffle products of words.
	In the rest of the section, we will aim at relating $\zeta^{\Dend}$ to other known objects, and in particular integral \AZVs{} and integral \MZVs{}. This will also solve the issue of convergence of dendriform zeta values that we have omitted so far.
	
	\section{From dendriform to usual \MZVs{}} \label{sec:dendri_to_usual}
	
	Unlike in section~\ref{sec:tridend_to_usual}, we are directly working with the adequate objects.
	
	\subsection{Arborified integrals zeta values}
	
	In the literature~\cite{Manchon_16,Ya20}, we can a find a definition of an \AZV{} for the integral representation of \MZV{}. We introduce a similar concept adapted to the trees we are considering:
	
	\begin{defi}[Arborification of the \MZVs{}]\label{def:azvs_integrals}
	Let $\Dendcv$ be the subspace of $\Dend(\{x,y\})$ generated by the tree $t=|$ and trees whose roots are decorated by $x$ and whose leaves are decorated by $y$.
	
		 We introduce the linear map $\zeta_{\Int}^T:\Dendcv  \rightarrow  \R$ such that for any tree $t$ of $\Dendcv$, it is defined by:
		\[
		\zeta_{\Int}^T(t)=\int_{u\in\Delta_t} \prod_{v\in\nu(t)} g_{d_v}(t_v)\dt_v,
		\]
		where $\Delta_t\subseteq \interff{0 1}^{|\nu(t)|}$ stands for:
		\[ 
		\left\{ \left(t_{v_1},\dots, t_{v_{|\nu(t)|}}\right) \,\middle|\, \forall (i,j)\in\IEM{1}{|\nu(t)|}^2, t_{v_i}> t_{v_j} \iff v_i\leq v_j  \right\},
		\] 
		where for any $(u,v)\in\nu(t)^2, u\leq v$ reading the tree as a Hasse diagram where the root is the minimum of the poset.
	\end{defi}
	
		\begin{Eg}\label{Eg:integral_representation_MZV}
		For instance:
		\[
		\zeta_{\Int}^T\left(\labYY{x}{y}{y}\right)=\int_{\substack{0< t_2< t_1< 1, \\ 0< t_3 < t_1< 1}} \frac{\dt_1}{t_1}\frac{\dt_2}{1-t_2}\frac{\dt_3}{1-t_3}.
		\]
	\end{Eg}
	
	\begin{Rq}
	 As in the tridendriform case, we are now working with Schroeder trees instead of non-planar ones and this has no impact on the arborified zeta values that we are studying. A new difference is that these trees have to be binary (in~\cite{Manchon_16,Ya20,clavier2020double} they are not), and their branching vertices do not have to be decorated by $y$, as in~\cite{clavier2020double}. However, one can easily check that the proofs of~\cite{clavier2020double} are still valid in this case: the planar binary Schroeder forests still form an $\{x,y\}$-operated algebra. All the results follow from this simple fact.
	\end{Rq}	

	We adapt the definition of the flattening to our binary trees using notation~\ref{Not:grafting_B}:
	\begin{defi}\label{def:flat_map_dend}
		We define the map $\flaten:\Dend(\{x,y\})\rightarrow \Wxy$ by $\flaten(|)=\emptyset$ and for any $\B^+_z(t_1,t_2)=t$ with $z\in\{x,y\}$ by:
		\[
		\flaten(t)=z\cdot(\flaten(t_1)\shuffle \flaten(t_2)).
		\]
	\end{defi}
	
	Then, we have the following properties of $\zeta^T_{\Int}$
	
	\begin{thm}[\cite{clavier2020double}]\label{thm:known_AZVs_dend}
		 We have $\flaten(\Dendcv)=\Wxycv$ and $\zeta^T_{\Int}$ factorises through the flattening: $\zeta^T_{\Int}=\zeta_{\Int} \circ \flaten$.
	\end{thm}
	
	Finally, we have an important property for our use:
	\begin{Lemme}\label{lem:flat_dend}
		The map $\flaten:\Dendcv \rightarrow \Wxycv$ is a morphism of dendriform algebras.
	\end{Lemme}
	\begin{proof}
		The proof is similar to the one of lemma~\ref{lem:flat_tridend}.
	\end{proof}
	
	\subsection{Relating the integral zetas}
	
		We recall the following three maps:	
		\begin{align*}
			i:\left\{ \begin{array}{rcl}
				\{x,y\} &\rightarrow & \Dend(\{x,y\}), \\
				x & \mapsto & \labY{x}, \\
				y & \mapsto &\labY{y},
			\end{array}\right. && L:\left\lbrace\begin{array}{rcl}
				\{x,y\} & \rightarrow & \A_{\{x,y\}}, \\
				x &\mapsto & \displaystyle\int_{0<t<1} g_x\dt, \\
				y &\mapsto & \displaystyle\int_{0<t<1} g_y\dt, \\
			\end{array} \right.
		\end{align*}
		and $j$ is the natural inclusion from $\{x,y\}$ into $\Wxy$. Then, let us recall (resp. define) the map of dendriform algebras $\zeta^{\Dend}: \Dend({\{x,y\}})\rightarrow \A_{\{x,y\}}$ (resp. $\Psi:\Dend(\{x,y\})\rightarrow \Wxy$) whose existence and unicity is given by the universal property of theorem~\ref{thm:univ_prop_dend} as we can see in diagrams~\ref{subfig:1freeness} (resp.~\ref{diag:2tridendfreeness}).
	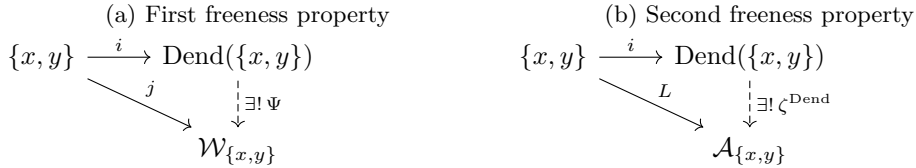
\begin{figure}[ht]
		\centering 
		\begin{subfigure}[t]{0.45\linewidth}
			\caption{First freeness property}
			\label{subfig:1freeness}
			\begin{tikzcd}
				\{x,y\} \arrow[rd, "j"] \arrow[r, "i"] & \Dend(\{x,y\}) \arrow[d ,dashed, "\exists!\,\Psi"] \\
				& \calW_{\{x,y\}}
			\end{tikzcd}
		\end{subfigure}
		\begin{subfigure}[t]{0.45\linewidth}
			\caption{Second freeness property}
			\label{subfig:2freeness}
			\begin{tikzcd}
				\{x,y\} \arrow[rd, "L"] \arrow[r, "i"] & \Dend(\{x,y\}) \arrow[d ,dashed, "\exists!\,\zeta^{\Dend}"] \\
				& \A_{\{x,y\}}
			\end{tikzcd}
		\end{subfigure}
		\caption{Diagrams for dendriform structures}
		\label{diag:dendri}
	\end{figure}
	
	\begin{thm}\label{thm:dend_diag}
		The following diagram commutes:
		\begin{tikzcd}
			\Dend(\{x,y\}) \arrow[d, "\flaten"]\arrow[r,"\zeta^{\Dend}"] & \A_{\{x,y\}} \\
			\Wxy \arrow[ur, "\zeta_{\FI}"] &
		\end{tikzcd}
	\end{thm}
	\begin{proof}
		We have shown in proposition~\ref{prop:formal_MZV_dend_map} that $\zeta_{\FI}$ is a dendriform morphism. Hence, $\zeta_{\FI}\circ \flaten$ and $\zeta^{\Dend}$ are both dendriform morphisms. Moreover, $\zeta_{\FI}\circ \flaten\circ i=L$. So, we deduce by the unicity of $\zeta^{\Dend}$ of  the universal property of theorem~\ref{prop:formal_MZV_dend_map} that $\zeta_{\FI}\circ \flaten=\zeta^{\Dend}$.
	\end{proof}
	\begin{Eg}
		For instance, considering the usual family of functions, we have:
		\[
		\zeta^{\Dend}\left(\labYY{x}{y}{y}\right)=\int_{\substack{0< t_2< t_1< 1, \\ 0< t_3 < t_1< 1}} g_xg_yg_y \dt_1\dt_2\dt_3.
		\]
		Then, evaluating this quantity gives the \AZVs of example~\ref{Eg:integral_representation_MZV}. 
	\end{Eg}
		
	Then, specializing the diagram of theorem~\ref{thm:dend_diag} to convergent trees we obtain
	\begin{Cor}\label{coro:relation_zetas_int}
		The following diagram commutes:
		\begin{tikzcd}
			\Dendcv \arrow[d, "\flaten"]\arrow[r,"\zeta^{\Dend}"] & \A_{\{x,y\}} \arrow[d,"\ev"]  \\
			\Wxycv \arrow[ur, "\zeta_{\FI}"] \arrow[r, "\zeta_{\Int}", swap] & \R
		\end{tikzcd}
	\end{Cor}
	In particular, we obtain that dendriform zeta values are integral \AZVs{}, and also linear combinations (with integer coefficients) 	of \MZVs{}, given by
	\[
	\forall t\in\Dendcv, \ev \circ \zeta^{\Dend}(t)=\zeta^T_{\Int}(t). 
	\]
	Finally, notice that this construction implies a new property for integral \AZVs{}
	\begin{Cor}\label{cor:azvs_dendri_alg}
		The map $\zeta_{\Int}^T:\Dendcv\rightarrow \R$ is a morphism of dendriform algebras. In particular, for any couple of tree $(t_1,t_2)\in(\Dendcv)^2:$
		\[
		\zeta_{\Int}^T(t_1\star t_2)=\zeta_{\Int}^T(t_1)\cdot \zeta_{\Int}^T(t_2).
		\]
	\end{Cor}
	\begin{proof}
		It is clear from the definition that $\Dendcv$ is a dendriform subalgebra of $\Dend(\{x,y\}).$ Then for any couple $(t_1,t_2)\in(\Dendcv)^2$ we have:
		\begin{align*}
			\zeta^T_{\Int}(t_1\star t_2) & = \zeta_{\Int}\circ\flaten(t_1\star t_2)\tag*{by corollary~\ref{coro:relation_zetas_int}} \\
			& = \zeta_{\Int}\circ\left(\flaten(t_1)\shuffle \flaten(t_2)\right) \tag*{by lemma~\ref{lem:flat_dend}} \\
			& = \zeta_{\Int}\circ\flaten(t_1)\zeta_{\Int}\circ\flaten(t_2)\tag*{since $\zeta_{\Int}$ is an algebra morphism for $\shuffle$} \\
			&=\zeta^T_{\Int}(t_1)\zeta_{\Int}^T(t_2)\tag*{by corollary~\ref{coro:relation_zetas_int}.} 
		\end{align*}
		Hence the corollary is proven.
	\end{proof}	
	
	So, we have solved the issue of finding an associative product generalizing the shuffle of words that is still compatible for integral \AZVs{}. Then, we sum up our results in diagram~\ref{diag:sumup} given in the conclusion.

\section{\SZVs{}} \label{sec:Shintani}

Let us start by introducing (mutlivariate) \SZVs{}. The following definition is borrowed from~\cite{matsumoto2003mordell}.
\begin{defi} \label{def:shintani}	
	\begin{itemize}
		\item Let $\Sigma_{n\times r}(\R_{\geq 0})$ be the set of $n\times r$ matrices with real non negatives entries, and with, at least, one non zero argument in each row and each column.
		\item Given a matrix $A = \{a_{ij}\}_{1\leq i \leq n, 1\leq j \leq r}\in \Sigma_{n\times r}(\R_{\geq 0})$ and a word $\omega\in \N*,$ we define the \emph{\Shin{} zeta function} associated to $A$ as
		\begin{equation}\label{shintani}
			\zeta_A(\omega) = \sum_{m_1\geq 1}\cdots \sum_{m_r\geq 1}(a_{11}m_1 + \cdots + a_{1r}m_r)^{-\omega_1}\cdots (a_{n1}m_1 + \cdots + a_{nr}m_r)^{-\omega_r}.
		\end{equation}
	\end{itemize}
\end{defi}
The sum \eqref{shintani} is absolutely convergent if $Re(\omega_i)>r$ for every $1\leq i \leq n.$ A better convergence result was obtained in~\cite{lopez2023milnor} but is rather technical and we will omit here.

Given a tree $t\in \Dendcv$ we will associate its arborified zeta value $\zeta^T_{\Int}(t)$ (definition~\ref{def:azvs_integrals}) to a \SZV{} $\zeta_{A_t}(\omega_t).$ For that, we need the following definitions, which are from~\cite{clavier2024generalisations}.

\begin{defi}
	Given  $t\in \Dend(\{x,y\})$ we define
	\begin{equation*} 
		\nu_y(t)=\{v\in \nu(t)\,|\, d(v)=y\}; \quad
		\nu_x(t)=\{v\in \nu(t)\,|\, d(v)=x\}.
	\end{equation*}

	A \emph{bifurcated vertex} is a vertex of $t$ that has more two descendants that are internal vertices. Let 
	\[\B(t) \coloneqq \{v\in \nu(t) \,|\, v\in \nu_y(t) \quad \text{or} \quad v\in \nu_x(t) \; \text{and} \; v \; \text{is a bifurcated vertex} \}.\]

	Let $(t,d_{t})$ be binary tree decorated by $\{x,y\}$ and $v\in \nu(t)$. A \emph{segment} ${s_v = (v_1,\dots,v_n = v)}$ is a non-empty path in $t$ such that $v_n = v\in \B(t),$ $v_i \not\in \B(t),$ for all ${i \in \{1,\dots, n-1\}}$ and $a(v_1) \in \B(t),$ where $a(v_1)$ is the direct ancestor of $v_1.$ We call $|s_v| = n$ the \emph{length} of $s_v.$ We say that $s_v \leq s_{v'}$ if $v\leq v'$ in $t.$
	
	We set $S(t) = \{s_v \,|\, v\in \B(t)\}.$
\end{defi}
In words, a segment $s_v$ of $t$ is a path that ends in an element of $\B(t)$ and contains every node decorated by $x$ that are not bifurcated between $v$ and it's first ancestor that is also in $\B(t).$ 

\medskip

We will associate the tree $t \in  \Dendcv$ to a pair $(A_t,\omega_t)$ with $A_t\in \Sigma_{|\B(t)|\times |\nu_y(t)|}(\R_{\geq 0})$ and  $\omega_t\in\calW_{\N^*,|\B(t)|}.$ 

\begin{defi} \label{def:matrix_word_Shintani}
 Let $t\in\Dendcv$ be a tree different from $|$. First enumerate the elements of $\B(t)$ from $1$ to $|\B(t)|$ and the element of $|\nu_y(t)|$ from $1$ to $|\nu_y(t)|$. Then the matrix ${A_t\coloneqq(a_{ij})_{(i,j)\in\{1,\dots,|\B(t)|\}\times\{1,\dots |\nu_y(t)|\}}}$ is defined by
  \begin{equation*}
   a_{ij}\coloneqq \begin{cases}
            1 \text{ if the vertex in $\B(t)$ numbered by $i$ is below the vertex in $\nu_y(t)$ numbered by $j$,}\\
            0 \text{ otherwise,}
           \end{cases}
  \end{equation*}
  for each $(i,j)\in\{1,\cdots,|\B(t)|\}\times\{1,\cdots |\nu_y(t)|\}$.
  
  For the word $\omega_t$, for any $i\in \{1,\cdots,|\B(t)|\}$, set the $i$-th letter of $\omega_t$ to be $|s_v|$, where $|s_v|$ is the is the length of $s_v.$
\end{defi}
	\begin{Eg}
	Consider the following tree \begin{equation}\label{eq:shintree}
		t=\begin{tikzpicture}[line cap=round,line join=round,>=triangle 45,x=0.25cm,y=0.25cm, baseline=(base)]
			\node (base) at (0,2) {};
			\draw [line width=.5pt][color=black] (0.,-1.)-- (0,0.);
			\draw [line width=.5pt][color=black] (0.,0.)-- (4,4.);
			\draw [line width=.5pt][color=black] (0.,0.)-- (-4,4.);
			\draw [line width=.5pt][color=black] (-2.,2.)-- (0,4.);
			\draw [line width=.5pt][color=black] (-0.5,3.5)-- (-1,4);
			\draw [line width=.5pt][color=black] (-3.5,3.5)-- (-3,4);
			\draw [line width=.5pt][color=black] (1.5,1.5)-- (1,2);
			\draw [line width=.5pt][color=black] (2.5,2.5)-- (2,3);
			\draw [line width=.5pt][color=black] (3.5,3.5)-- (3,4.);
			\draw (0,0) node[above][color=black]{\scriptsize{$x$}};
			\draw (-2,2) node[above][color=black]{\scriptsize{$x$}};
			\draw (-0.5,3.5) node[above][color=black]{\scriptsize{$y$}};
			\draw (-3.5,3.5) node[above][color=black]{\scriptsize{$y$}};
			\draw (1.5,1.5) node[above][color=black]{\scriptsize{$y$}};
			\draw (2.5,2.5) node[above][color=black]{\scriptsize{$x$}};
			\draw (3.5,3.5) node[above][color=black]{\scriptsize{$y$}};
			\filldraw[color=green] (0,0) circle (1pt) ;
			\filldraw[color=blue] (3.5,3.5) circle (1pt) ;
			\filldraw[color=red] (-3.5,3.5) circle (1pt) ;
			\filldraw[color=yellow] (-0.5,3.5) circle (1pt) ;
			\filldraw[color=cyan] (-2,2) circle (1pt) ;
			\filldraw[color=brown] (1.5,1.5) circle (1pt) ;
		\end{tikzpicture},
	\end{equation}
	from the construction above we have 
	\begin{equation*}
		A_t = \begin{pNiceMatrix}[first-col, first-row] 
			& \circlecoloured{red} & \circlecoloured{yellow} & \circlecoloured{blue} & \circlecoloured{brown}\\
			\circlecoloured{red} & 1 & 0 & 0 & 0\\
			\circlecoloured{yellow} & 0 & 1 & 0 & 0\\
			\circlecoloured{cyan} & 1 & 1 & 0 & 0\\
			\circlecoloured{blue} & 0 & 0 & 1 & 0\\
			\circlecoloured{brown} & 0 & 0 & 1 & 1\\
			\circlecoloured{green} & 1 & 1 & 1 & 1
		\end{pNiceMatrix} \quad \text{and} \quad \omega_t = 111211.
	\end{equation*}
\end{Eg}
\begin{Rq}
	Permuting the rows of the matrix does not change the value of the \SZV{}.
\end{Rq}

The main result of this section is the existence of a series representation of $\ev\circ\zeta^{\Dend}(t)$. We will omit the proof since it is an adaptation of~\cite[Theorem 1.19]{clavier2024generalisations}. Indeed, the only difference is that bifurcated vertices can now be decorated by $x$. It is easy to see that this does not change the proof. It does change the matrix $A_t$ we built above, hence we state this result with \SZVs{} rather than with conical zeta values.

\begin{thm} \label{thm:dend_zeta_series}
	For any tree $t\in \Dendcv$ we have
	\begin{equation}\label{eq:sum_thm}
		\ev\circ\zeta^{\Dend}(t)=\zeta_{\Int}^T(t) = \sum_{\substack{n_v\geq 1\\v\in \nu_y(t)}}\prod_{v\in \B(t)}\left(\sum_{\substack{v'\in \nu_y(t)\\v'\succeq v}}n_{v'}\right)^{-|s_v|}
		=\zeta_{A_t}(\omega_t).
	\end{equation}
	where the matrix $A_t$ and the word $\omega_t$ are constructed in definition~\ref{def:matrix_word_Shintani}.
\end{thm}

	\begin{Eg}
		From the tree in equation~\eqref{eq:shintree} we have
		\begin{equation*}
			\zeta(t) = \sum_{m_1\geq 1}\sum_{m_2\geq 1}\sum_{m_3\geq 1}\sum_{m_4\geq 1}\sum_{m_5\geq 1}\dfrac{1}{m_1 m_2 (m_1+m_2)m_3^2(m_3+m_4)(m_1+m_2+m_3+m_4+m_5)}.
		\end{equation*}
	\end{Eg}

	\begin{Rq}
		We found in this section a class of \SZVs{} that are convergent but do not follow the sufficient convergent condition of~\cite{matsumoto2003mordell}. Furthermore, this result (together with theorem~\ref{thm:known_AZVs_dend} also implies that a family of \SZVs{} can be written as linear combinations of \MZVs{} with rational coefficients.
	\end{Rq}

\section*{Conclusion}

	\addcontentsline{toc}{section}{Conclusion}

The diagram~\ref{diag:sumup} below summarizes all the known results (with the $\mathfrak{s}$ map is the Kontsevitch map detailed in the introduction) and the ones from our construction in only one picture.
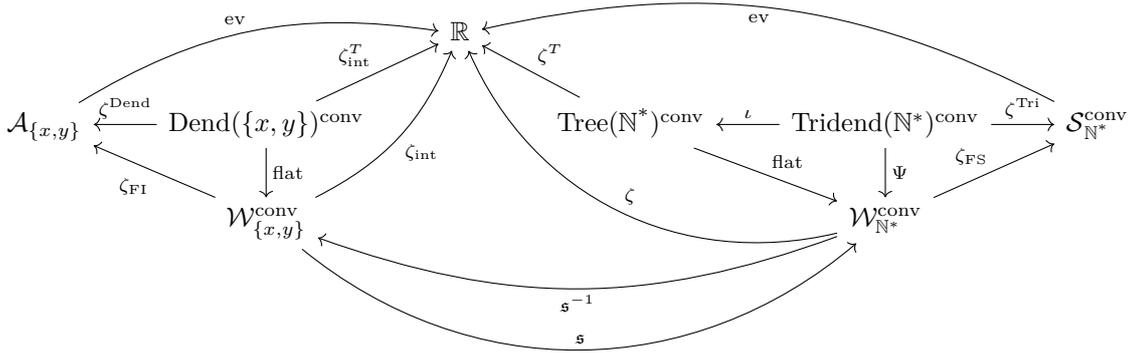
\begin{figure}[ht]
	\centering
	\begin{tikzcd}
	& & \R & & & \\
	\A_{\{x,y\}} \arrow[urr,"\ev",bend left=20] & \Dendcv \arrow[l, "\zeta^{\Dend}", swap] \arrow[d, "\flaten"] \arrow[ur,"\zeta^T_{\rm int}"] &  & \Treecv \arrow [ul,"\zeta^T", swap] \arrow[rd, "\flaten"] & \Tridend(\N^*)^{\rm conv} \arrow[l, "\iota", swap] \arrow[r,"\zeta^{\rm Tri}"] \arrow[d, "\Psi"] & \Suumcv \arrow[ulll, "\ev", bend right=20]  \\
	& \Wxycv \arrow[rrr, "\mathfrak{s}", bend right = 40] \arrow[uur, "\zeta_{\Int}"swap, , bend right=20] \arrow[ul,"\zeta_{\FI}"] &  &  & \Wcv \arrow[lll, "\mathfrak{s}^{-1}", bend left=20] \arrow[ur, "\zeta_{\FS}"] \arrow[uull, "\zeta", swap, bend left=40] &  
	\end{tikzcd}
	\caption{Picture of the general situation}
	\label{diag:sumup}
\end{figure}

With our construction, we have:
\begin{itemize}
	\item introduced two intermediate objects, namely formal series and formal integrals, in order to factorize $\zeta$ and $\zeta_{\Int}$ with respectively a tridendriform morphism and a dendriform morphism;
	\item built two new generalisations of $\zeta_{\Int}$ and $\zeta$ using the universal properties of Schroeder trees in the categories of dendriform and tridendriform algebras and related these generalisations to usual \MZVs{} and arborified zeta values $\zeta^T_{\Int}$ and $\zeta^T$;
	\item built two associative products $\star$ for binary trees and $*$ for \Sch{} trees (in a non-inductive way) such that $\zeta_{\Int}^T$ and $\zeta^T$ are respectively morphisms using universal properties of the free dendriform and free tridendriform algebras;
	\item shown that any iterated integral $\ev\circ\zeta^{\Dend}(t)$ can be written as a multiple series which is a \Shin{}'s \MZVs{}.
\end{itemize}

From this work some new perspectives arise, which could be tackled in future projects:
\begin{itemize}
	\item Could we find an analogous of the Kontsevitch's map $\mathfrak{s}$ for trees such that the diagram is still commutative ? 
	
	At first it may be not so simple. For instance, 
	\[
	\zeta^T_{\Int}\left(\labYY{x}{y}{y}\right)=\sum_{\substack{n\geq 1 \\ m\geq 1}}\frac{1}{n\cdot m \cdot (n+m)}=2\zeta(2~1),
	\]
	but one cannot find a \emph{single} tree $t \in \Tree(\N^*)$ giving rise to the above sum as the tree should have three vertices $v_1,v_2$ and $v_3$ satisfying $v_1<v_3$ and $v_2<v_3$. So if it exists, it should be \begin{tikzpicture}[x=0.3cm,y=0.3cm]
		\draw (0,0)--(1,1);
		\draw (1,1)--(2,0);
		\filldraw (0,0) circle (1pt);
		\filldraw (1,1) circle (1pt);
		\filldraw (2,0) circle (1pt);
	\end{tikzpicture} but this obviously not a tree. 
	
The recent work~\cite{Fan25} proposes a new generalisation of Kontsevitch's map that has some interesting properties. Seeing how this map fits into the picture displayed in figure~\ref{diag:sumup} could solve this problem.
	\item We noticed with this construction that we are able to use universal properties of free dendriform and tridendriform algebras. Could we do this game for other type of algebras ? 
	\item Both free dendriform and free tridendriform algebras have a coalgebra structure. Is it possible to interpret these coalgebra structures for \MZVs{} ? 
\end{itemize}

	\bibliography{biblio_MZV_Fin}
	\bibliographystyle{plain}
	
\end{document}